\documentclass[12pt]{article}
\usepackage[charter]{mathdesign}
\usepackage{color}
\usepackage{graphicx,enumerate}
\usepackage{amsmath,amsthm}
\usepackage{setspace}
\usepackage{enumitem}
\usepackage{pdfpages}
\setlist{noitemsep}

\usepackage[procnames]{listings}

\definecolor{Code}{rgb}{0,0,0}
\definecolor{Decorators}{rgb}{0.5,0.5,0.5}
\definecolor{Numbers}{rgb}{0.5,0,0}
\definecolor{MatchingBrackets}{rgb}{0.25,0.5,0.5}
\definecolor{Keywords}{rgb}{0,0,1}
\definecolor{self}{rgb}{0,0,0}
\definecolor{Strings}{rgb}{0,0.63,0}
\definecolor{Comments}{rgb}{0,0.63,1}
\definecolor{Backquotes}{rgb}{0,0,0}
\definecolor{Classname}{rgb}{0,0,0}
\definecolor{FunctionName}{rgb}{0,0,0}
\definecolor{Operators}{rgb}{0,0,0}
\definecolor{Background}{rgb}{0.98,0.98,0.98}
\newtheorem{theorem}{Theorem}[section]
\newtheorem{lemma}[theorem]{Lemma}

\newtheorem{corollary}[theorem]{Corollary}
\newtheorem{proposition}[theorem]{Proposition}

\theoremstyle{definition}
\newtheorem{definition}[theorem]{Definition}

\newcommand{\mcal}[1]{\ensuremath{\mathcal{#1}}}
\newcommand{\GF}[1]{\ensuremath{\mathrm{GF}(#1)}}

\newcommand{\delete}{\backslash}
\newcommand{\contract}{/}

\newcommand{\dash}{\nobreakdash-\hspace{0pt}}

\DeclareMathOperator{\si}{si}

\newcommand{\hydra}{\mathbb{H}}
\DeclareMathOperator{\assoc}{Asc}
\newcommand{\Q}{\mathbb{Q}}
\DeclareMathOperator{\core}{Core}

\usepackage[colorlinks,linkcolor=blue,anchorcolor=black,citecolor=blue,filecolor=blue,menucolor=blue,runcolor=blue,urlcolor=blue]{hyperref}

\reversemarginpar

\usepackage{booktabs}

\begin{document}
  \title{Computer-verification of the structure of some classes of fragile matroids}
  \author{Carolyn Chun\footnote{School of Information Systems, Computing and Mathematics, Brunel University, United Kingdom. Email: \url{carolyn.chun@brunel.ac.uk}} \and Deborah Chun\footnote{Department of Mathematics, West Virginia University Institute of Technology, United States. Email: \url{deborah.chun@mail.wvu.edu}} \and Benjamin Clark\footnote{School of Mathematics, Statistics and Operations Research, Victoria University of Wellington, New Zealand. \url{{benjamin.clark,dillon.mayhew,geoff.whittle}@msor.vuw.ac.nz}} \and Dillon Mayhew\footnotemark[3] \and Geoff Whittle\footnotemark[3] \and Stefan H.M. van Zwam\footnote{Department of Mathematics, Louisiana State University, United States. Email: \url{svanzwam@math.lsu.edu}. Supported by the National Science Foundation, Grant No. 1161650/1501985.}}

%
%
%
%
%
\date{\today}

\maketitle


\section{Introduction}
This technical report accompanies the papers \cite{CMWZ10,CMWZ,CCMZ}. It contains the computations necessary to verify some of the results claimed in those papers. We start by describing those results. To do so, we need the concept of \emph{gluing wheels to triangles}. We will delay the formal definitions until the next section, and point to Figure \ref{fig:growfans} to get an intuitive idea.

\begin{figure}[htbp]
	\centering
		\includegraphics[scale=0.8]{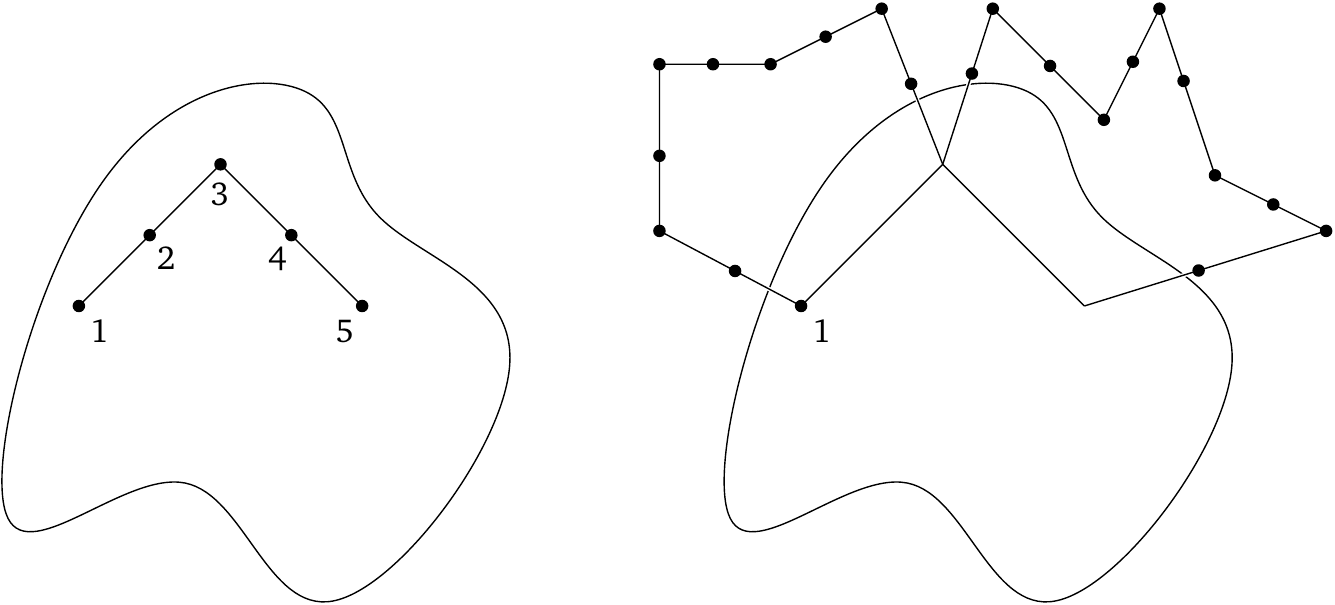}
	\caption{Gluing wheels onto the triangles $(1,2,3), (3,4,5)$. Elements $2$ and $4$ are always removed; elements $1,3,5$ may or may not be removed.}
	\label{fig:growfans}
\end{figure}

Let $\mathcal{N}$ be a set of matroids. A matroid $M$ is \emph{$\mathcal{N}$-fragile} if, for each $e \in E(M)$, at most one of $M\delete e, M\contract e$ has a minor isomorphic to a member of $\mathcal{N}$. If $M$ itself does have such a minor, then $M$ is \emph{strictly} $\mathcal{N}$-fragile.

\subsection{Fano-fragile matroids}
We say a matroid is \emph{Fano-fragile} if it is binary and is strictly $\{F_7, F_7^*\}$-fragile. In \cite{CMWZ10}, Chun, Mayhew, Whittle, and Van Zwam determine the structure of Fano-fragile matroids. There are several classes; in this report we deal with two of them. For the first, we look at matroids containing $R_{10}$ as a minor. It turns out that there is a unique (up to isomorphism) Fano-fragile single-element extension of $R_{10}$, which we will call $N_{11}$. We assume $N_{11}$ has a triangle $(0,10,4)$ such that $N_{11} / 10$ has a minor in $\{F_7, F_7^*\}$.

\begin{theorem}\label{thm:N11}
	Let $M$ be a 3-connected, Fano-fragile matroid having a minor isomorphic to $N_{11}$. Then $M$ is isomorphic to a matroid obtained from $N_{11}$ by gluing a wheel onto the triangle $(0,10,4)$.
\end{theorem}

Note that Truemper \cite[p. 300]{TruVI} essentially stated this result, although without proof, and for a slightly more restricted class of matroids than the one we consider.

Let $(2,6,9), (8,0,9), (1,7,9)$ be triangles of $F_7$, with $\{6,0,7\}$ an independent set and $\{2,8,1\}$ a circuit. The matroid $N_{12}$ is obtained from $F_7$ by gluing rank-3 wheels to these three triangles, while deleting the elements $\{0,6,7,9\}$ (a more formal definition follows below). We prove the following:

\begin{theorem}\label{thm:N12}
	Let $M$ be a 3-connected, Fano-fragile matroid having a minor isomorphic to $N_{12}$. Then $M$ is isomorphic to a matroid obtained from $F_{7}$ by gluing wheels to the triangles $(2,6,9), (8,0,9), (1,7,9)$.
\end{theorem}

These matroids have previously appeared in work by Kingan and Lemos \cite{KL02} as the family $\mathcal{F}_1(m,n,r)$ (although in their family the elements $1,2,8,9$ are never deleted).

Moreover, we find representations of the matroids in Theorems \ref{thm:N11} and \ref{thm:N12} in terms of \emph{grafts}. Again, the graft representation for Theorem \ref{thm:N11} can be found in \cite[p.300]{TruVI}.

\subsection{$\{U_{2,5}, U_{3,5}\}$-fragile matroids}
Let $\mathcal{M}_5$ be the set of strictly $\{U_{2,5},U_{3,5}\}$-fragile matroids that are representable over the partial field $\mathbb{H}_5$ (see \cite{PZ08conf} for a definition). Let $\mathcal{M}_2$ be the set of strictly $\{U_{2,5},U_{3,5}\}$-fragile matroids that are representable over the partial field $\mathbb{U}_2$ (first introduced in \cite{Sem97}). In \cite{CMWZ}, Clark, Mayhew, Whittle, and Van Zwam determine the structure of the matroids in $\mathcal{M}_5$ and in $\mathcal{M}_2$. Most of that paper is taken up by a study of the structure of matroids having one of three specific matroids, $X_8, Y_8, Y_8^*$, as a minor. In this report we determine the structure of the remaining matroids in the class. We have the following:

\begin{theorem}\label{thm:H5characterization}
	Let $M' \in \mathcal{M}_5$ be 3-connected. Then $M'$ is isomorphic to a matroid $M$ for which one of the following holds:
	\begin{enumerate}
		\item $M$ has one of $X_8, Y_8, Y_8^*$ as a minor;
		\item $M \in \{U_{2,6}, U_{4,6}, P_6, M_{9,9}, M_{9,9}^*\}$;
		\item $M$ or $M^*$ can be obtained from $U_{2,5}$ (with groundset $\{a,b,c,d,e\}$) by gluing wheels to $(a,c,b), (a,d,b), (a,e,b)$;
		\item $M$ or $M^*$ can be obtained from $U_{2,5}$ (with groundset $\{a,b,c,d,e\}$) by gluing wheels to $(a,b,c), (c,d,e)$;
		\item $M$ or $M^*$ can be obtained from $M_{7,1}$ by gluing a wheel to $(1,3,2)$.
	\end{enumerate}
\end{theorem}

The matroids $X_8$, $Y_8$, $M_{9,9}$, and $M_{7,1}$ will be defined in Section \ref{sec:U25}. By considering which of these matroids are in $\mathcal{M}_2$, we immediately obtain

\begin{corollary}\label{thm:U2characterization}
	Let $M' \in \mathcal{M}_2$ be 3-connected. Then $M'$ is isomorphic to a matroid $M$ for which one of the following holds:
	\begin{enumerate}
		\item $M$ has one of $X_8, Y_8, Y_8^*$ as a minor;
		\item $M \in \{M_{9,9}, M_{9,9}^*\}$;
		\item $M$ or $M^*$ can be obtained from $U_{2,5}$ (with groundset $\{a,b,c,d,e\}$) by gluing wheels to $(a,c,b), (a,d,b), (a,e,b)$;
		\item $M$ or $M^*$ can be obtained from $U_{2,5}$ (with groundset $\{a,b,c,d,e\}$) by gluing wheels to $(a,b,c), (c,d,e)$;
		\item $M$ or $M^*$ can be obtained from $M_{7,1}$ by gluing a wheel to $(1,3,2)$.
	\end{enumerate}
\end{corollary}

\subsection{This report}
The report is built up as follows. In Section \ref{sec:fans}, we describe definitions and results from \cite{CCMZ}. It contains a formal definition of what it means to glue wheels to triangles, and results that reduce the proofs of the theorems above to finite case checks. In Sections \ref{sec:N11} and \ref{sec:U25}, we prove the theorems above, as well as a few additional lemmas. The computational aspects of those proofs are relegated to the appendix. The computations were carried out in Version 6.5 of SageMath \cite{sage}, in particular making use of the \emph{matroids} component \cite{sage-matroid}.

\section{On fan-extensions and gluing wheels}\label{sec:fans}

The content of this section comes directly from \cite{CCMZ}. We start with some definitions regarding fans. Recall that a fan of the matroid $M$ is a sequence of
distinct elements, $(e_{1},\ldots, e_{m})$, such that $m\geq 3$ and
\[
\{e_{1},e_{2},e_{3}\},\{e_{2},e_{3},e_{4}\},\ldots,
\{e_{m-2},e_{m-1},e_{m}\}
\]
is an alternating sequence of triangles and triads.
If $\{e_{1},e_{2},e_{3}\}$ is a triangle, then the elements
with odd indices are \emph{spoke elements}, and the
elements with even indices are \emph{rim elements}.
These labels are reversed when $\{e_{1},e_{2},e_{3}\}$ is a triad.
We blur the distinction between a fan and the underlying
set of elements, so when we say that two fans are disjoint, we mean
that their underlying sets are disjoint.



\begin{definition}\label{def:fanlengthening}
  Let $M$ and $N$ be 3-connected matroids, and let $(e_1, \ldots, e_n)$ be a fan of $M$, with $n \geq 4$. We say $M$ was obtained from $N$ by a \emph{fan-lengthening move} if one of the following holds:
  \begin{itemize}
    \item $N = M\delete e_1$ (so $e_1$ is a spoke element);
    \item $N = M\delete e_n$ (so $e_n$ is a spoke element);
    \item $N = M\contract e_1$ (so $e_1$ is a rim element);
    \item $N = M\contract e_n$ (so $e_n$ is a rim element);
    \item $n \geq 5$ and $N = M \contract e_i \delete e_{i+1}$ for some $1 \leq i \leq n-1$ (so $e_i$ is a rim element and $e_{i+1}$ is a spoke element).
    \item $n \geq 5$ and $N = M \contract e_{i+1} \delete e_i$ for some $1 \leq i \leq n-1$ (so $e_i$ is a spoke element and $e_{i+1}$ is a rim element).
  \end{itemize}
\end{definition}
In each case, $\{e_1, \ldots, e_n\} \cap E(N)$ is a fan of $N$. We say a sequence $F = (f_1, \ldots, f_m)$ is \emph{consistent} with a sequence $G = (g_1, \ldots, g_n)$ if $F$ appears as a (not necessarily contiguous) subsequence of either $G$ or its reversal.

\begin{definition}\label{def:coveringfam}
  Let $N$ be a 3-connected matroid on at least 4 elements, and let $\mathcal{F}_N$ be a collection of pairwise disjoint fans in $N$. If $M$ has $N$ as a minor, and $\mathcal{F}$ is a collection of fans of $M$, then we say $\mathcal{F}$ is a \emph{covering family} of $M$ (relative to $N$ and $\mathcal{F}_N$) if
  \begin{itemize}
    \item The fans in $\mathcal{F}$ are pairwise disjoint;
    \item $|\mathcal{F}| = |\mathcal{F}_N|$;
    \item For every $F_N \in \mathcal{F}_N$, there is a fan $F \in \mathcal{F}$ such that $F_N$ is consistent with $F$;
    \item Every element in $E(M) - E(N)$ is contained in some $F \in \mathcal{F}$.
  \end{itemize}
\end{definition}

\begin{definition}\label{def:fanextension}
  Let $N$ be a 3-connected matroid, and $\mathcal{F}_N$ a collection of pairwise disjoint fans of $N$. We define $\mathcal{M}_N$ to be the smallest family of matroids satisfying:
  \begin{itemize}
    \item $N \in \mathcal{M}_N$;
    \item If $M \in \mathcal{M}_N$, $\mathcal{F}$ is a covering family of $M$, and $F \in \mathcal{F}$, then each matroid obtained from $M$ by a fan-lengthening move on $F$ is in $\mathcal{M}_N$.
  \end{itemize}
  We say that $\mathcal{M}_N$ is the set of \emph{fan-extensions} of $N$ (relative to $\mathcal{F}_N$).
\end{definition}
Note that each fan-extension has, by construction, a covering family. Note also that covering families can pick up elements from $E(N)$ that aren't contained in fans in $\mathcal{F}_N$, usually because these elements have a ``choice'' of fans to which they might belong.

The main result that Chun, Chun, Mayhew, and Van Zwam \cite{CCMZ} prove is the following:

\begin{lemma}[{\cite[Theorem 6.10]{CCMZ}}]\label{lem:fanlemma}
  Let $\mathcal{M}$ be a set of matroids that is closed under minors and isomorphism. Let $N \in \mathcal{M}$ be a 3-connected matroid such that $|E(N)| \geq 4$ and N is neither a wheel nor a whirl. Assume that any member of $\mathcal{M}$ with $N$ as a minor is 3-connected up to series and parallel sets. Let $\mathcal{F}_N$ be a family of pairwise disjoint fans of $N$. If there is a 3-connected matroid in $\mathcal{M}$ with $N$ as a minor that is not a fan-extension of $N$ relative to $\mathcal{F}_N$, then there exists such a matroid, $M$, satisfying $|E(M)| - |E(N)| \leq 2$.
\end{lemma}

Not all covering families can be obtained by fan-extensions, but the following proposition shows that this is the case for all covering families in our applications. A \emph{fan-shortening move} is the reverse of a fan-lengthening move.

\begin{proposition}[{\cite[Proposition 2.4]{CCMZ}}]
\label{court}
Assume there is no fan, $F$, in $N$ such that two distinct fans in $\mcal{F}_{N}$
(considered as unordered sets) are contained in $F$.
Let $M$ be a $3$\dash connected matroid with $N$ as a proper minor,
and assume that every minor of $M$ that has $N$ as a minor is $3$\dash connected
up to series and parallel sets.
Let \mcal{F} be an arbitrary covering family of $M$.
Then \mcal{F} admits a fan-shortening move.
\end{proposition}

Now that we have reduced checking if a class of matroids contains only fan-extensions to a finite case check, it is time to relate fan-extensions to the more structural process of gluing wheels to triangles. We denote the generalized parallel connection of $M_1$ and $M_2$ along flat $T$ (which is a modular flat of $M_2$) by $M_1 \boxtimes_T M_2$. If $M_{1}\boxtimes_{T_{2}}M_{2}$
and $M_{1}\boxtimes_{T_{3}}M_{3}$ are both defined, then
it follows easily from the definition that
$(M_{1}\boxtimes_{T_{2}}M_{2})\boxtimes_{T_{3}}M_{3}$ and
$(M_{1}\boxtimes_{T_{3}}M_{3})\boxtimes_{T_{2}}M_{2}$ are defined and equal.
Note that a triangle in a simple binary matroid is a modular flat.

\begin{definition}\label{def:wheelglue}
  Let $N$ be a matroid, $t$ an integer, and $T_i = (a_i,b_i,c_i)$ triangles of $N$, for $i = 1, \ldots, t$. We say $M$ is obtained from $N$ by \emph{gluing wheels to $T_1, \ldots, T_t$} if $M$ can be obtained in the following way. For each $i$, let $n_i \geq 2$ be an integer, and $W_i$ be a wheel of rank $n_i$ with one triangle labeled $(a_i, b_i, c_i)$, where $a_i$ and $c_i$ are spoke elements. Define $N_0 := N$ and, for $i = 1, \ldots, t$, $N_i := N_{i-1} \boxtimes_{T_i} W_i$. Let $X \subseteq T_1 \cup \cdots \cup T_t$ be such that $b_i \in X$ for all $i$ unless $b_i = a_j$ or $b_i = c_j$ for some $j$. Other elements may or may not be in $X$. Now $M = N_t\delete X$.
\end{definition}

%
%

Finally, we need to find the matroid $N$ from the previous definition. This matroid usually has lower rank than the matroid to which Lemma \ref{lem:fanlemma} is applied. The following (lengthy) definition tells us how to construct this matroid, which we call the \emph{core} of $N$:

\begin{definition}\label{def:core}
	Let $N$ be a matroid, represented over a field $\mathbb{F}$, and $\mathcal{F} = \{F_1, \ldots, F_t\}$ a family of pairwise disjoint fans in $N$. Consider $N$ as a restriction of some projective geometry $P$. Set $N_0 := N$.
	
	For $i = 1, \ldots, t$, let $F_i = (e_1, \ldots, e_m) \in \mathcal{F}$.	We obtain $N_i$ from $N_{i-1}$ as follows. If $e_{1}$ is a spoke element, let $a_i' := e_1$. If $e_{1}$ is a rim element, then let $a_i'$ be the point of $P$ that is in the closure of $E(N)-F_{i}$ and of $\{e_{1},e_{2}\}$.
	Similarly, if $e_m$ is a spoke element, then let $c_i' := e_m$, 
	and otherwise let $c_{i}'$ be the point of $P$ that lies 
	in the closure of $E(N)-F_{i}$ and $\{e_{m-1},e_{m}\}$.
	Finally, let $R$ be the set of rim elements in $F_{i}$, and let
	$b_{i}'$ be the point in the closure of $R$ and $E(N)-F_{i}$. Now add points $a_i, b_i, c_i$ to the matroid $P$ such that $\{a_i,a_i'\}, \{b_i,b_i'\}, \{c_i,c_i'\}$ are parallel pairs, and let $N_i$ be the matroid obtained from $N_{i-1}$ by adding $\{a_i,b_i,c_i\}$.

	Let $L_{i}=\{a_{i},b_{i},c_{i}\}$, and let $L$ be
	$\bigcup_{i=1}^t L_{i}$. Let $S$ be the elements of $E(N)$ that are in parallel with some $a_i$ or $c_i$.
	We define
	\[
	  \core(N) := N_t \delete F_1 \delete \cdots \delete F_t \delete S.
	\]	
	 We call $\core(N)$ the \emph{core of $N$ relative to $\mathcal{F}$}.
\end{definition}

Some remarks are in place. Note that $L_i = \{a_i, b_i, c_i\}$ is a triangle for each $i$. 
Note also that we have defined the core relative to a representation of $N$.
In fact, any two representations of $N$ will lead to the same matroid $\core(N)$,
but we will not make use of this fact.

\begin{lemma}[{\cite[Lemma 3.5]{CCMZ}}]\label{lem:wheelglue}
	Let $N$ be a $3$\dash connected representable matroid,
	where $|E(N)|\geq 4$.
	Let $\mcal{F}_{N}$ be a pairwise disjoint family of fans in $N$.
	Assume there is no fan, $F$, in $N$ such that two distinct fans in $\mcal{F}_{N}$
	(considered as unordered sets) are contained in $F$.
	Assume that $M$ is a $3$\dash connected matroid with $N$ as a minor, and
	every minor of $M$ with $N$ as a minor is $3$\dash connected up to series and
	parallel sets.
	Let \mcal{F} be a covering family in $M$ (relative to $\mcal{F}_{N})$, so that
	$M$ is a fan-extension of $N$.
	There exists a pairwise disjoint family of fans, $\mcal{F}^{+}$, in $N$, such that
	$\mcal{F}^{+}$ is a covering family of $\mcal{F}_{N}$, and moreover,
	we can relabel $M$ in such a way that it is obtained by gluing wheels to $\core(N)$
	(where $\core(N)$ is defined relative to $\mcal{F}^{+}$), and
	\mcal{F} is enclosed in the family of canonical fans associated with the gluing operation.
\end{lemma}

\section{Fano-fragile matroids}\label{sec:N11}

\begin{figure}[htbp]
	\centering
		\includegraphics[scale=0.8]{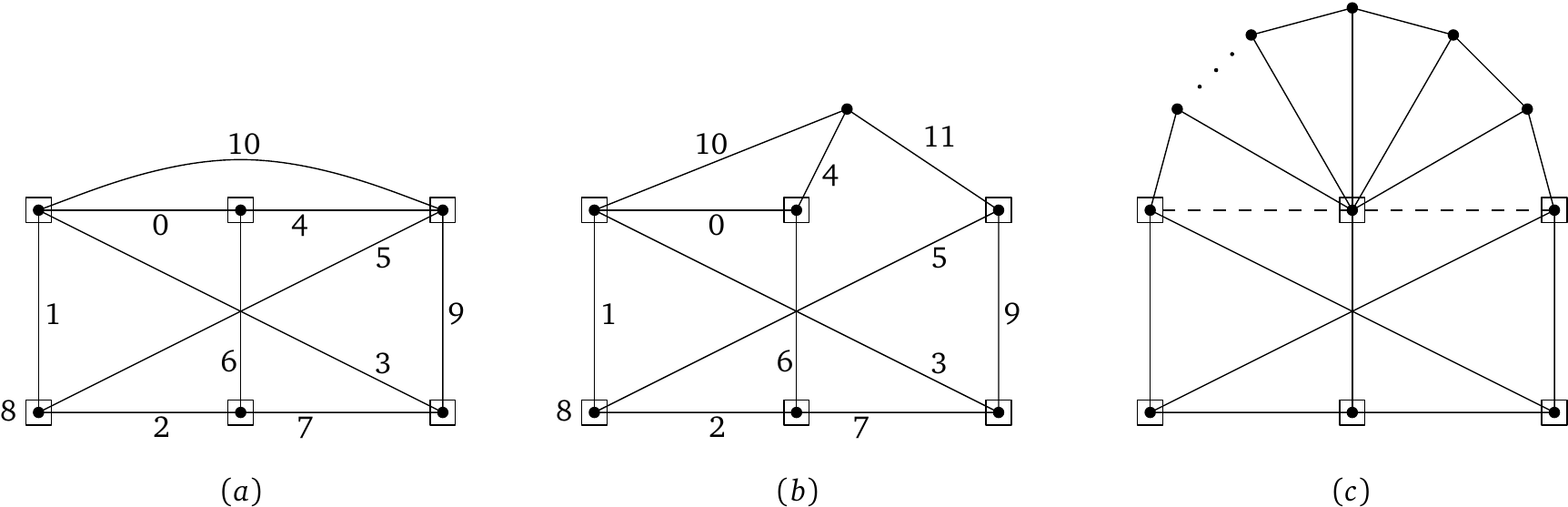}
	\caption{Graft representations of (a) $N_{11}$; (b)  $N_{11}^+$; (c) a matroid obtained from $N_{11}$ by gluing a wheel to $(0,10,4)$.}
	\label{fig:N11graft}
\end{figure}

In this section we will prove Theorems \ref{thm:N11} and \ref{thm:N12}. Consider the matroid $N_{11}$. A definition can be found on Page 2 of the Supporting Computations appendix. We can interpret $N_{11}$ as a graft (see \cite[p. 386]{ox2}; we use squares to indicate graft vertices), as shown in Figure \ref{fig:N11graft}(a). As verified in the appendix, $N_{11}$ has (up to isomorphism) a unique 3-connected, Fano-fragile coextension, which we denote by $N_{11}^+$, and no 3-connected, Fano-fragile extensions. See Figure \ref{fig:N11graft}(b). We repeat Theorem \ref{thm:N11}.

\begin{theorem}\label{thm:N11repeat}
	Let $M$ be a 3-connected, Fano-fragile matroid having a minor isomorphic to $N_{11}$. Then $M$ is isomorphic to a matroid obtained from $N_{11}$ by gluing a wheel to the triangle $(0,10,4)$.
\end{theorem}

\begin{proof}
	This is clearly true for $N_{11}$ and $N_{11}^+$. By the Splitter Theorem, we may assume that $M$ has $N_{11}^+$ as a minor. It follows from Lemma A.1 that $M$ is a fan-extension of $N_{11}^+$ with respect to the fan $(0,10,4,11)$. It is clear that, up to relabeling elements, $\core(N_{11}^+)$ equals $N_{11}$. The desired result now follows from Lemma \ref{lem:wheelglue} with $N = N_{11}^+$ and $\mathcal{F}_N = \mathcal{F}^+ = \mathcal{F}$.
\end{proof}

The following is now obvious:

\begin{corollary}\label{cor:N11graft}
	If $M$ is a 3-connected, Fano-fragile matroid having a minor isomorphic to $N_{11}$, then $M$ has a graft representation as in Figure \ref{fig:N11graft}(c), where each of the dashed edges may or may not be present.
\end{corollary}

\begin{figure}[htbp]
	\centering
		\includegraphics[scale=0.9]{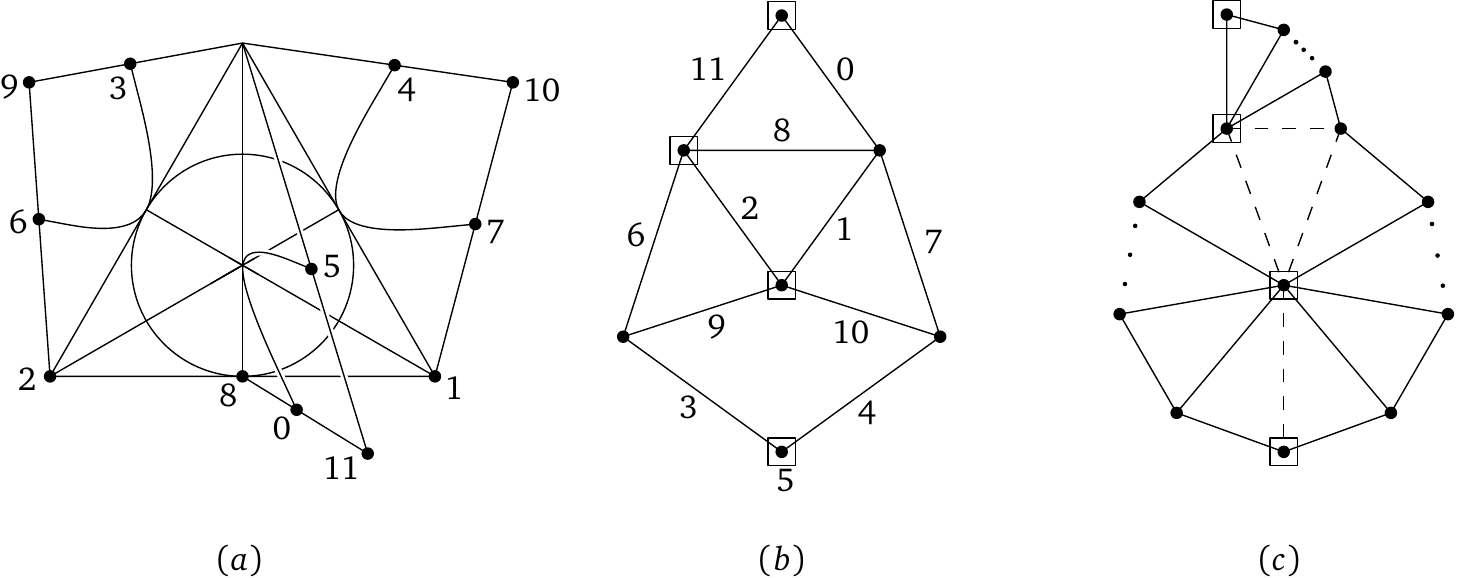}
	\caption{(a) Geometric representation of $N_{12}$; (b) graft representation of $N_{12}$; (c) graft representation of $M$ obtained from $F_7$ by gluing a wheel to $(2,6,9), (8,0,9), (1,7,9)$.}
	\label{fig:N12graft}
\end{figure}

In Figure \ref{fig:N12graft} we show a geometric and a graft representation of the matroid $N_{12}$. We leave it to the reader to confirm that these pictures are consistent with the matrix representation in the appendix. We assume the elements of $F_7$ are labeled as they would be in $N_{12} \contract \{3,4,5\} \delete \{10,11\}$. We repeat Theorem \ref{thm:N12}.

\begin{theorem}\label{thm:N12repeat}
	Let $M$ be a 3-connected, Fano-fragile matroid having a minor isomorphic to $N_{12}$. Then $M$ is isomorphic to a matroid obtained from $F_{7}$ by gluing wheels to the triangles $(2,6,9), (8,0,9), (1,7,9)$.
\end{theorem}

\begin{proof}
	It follows from Lemma A.2 that $M$ is a fan-extension of $N_{12}$ with respect to $\mathcal{F} = \{(2, 6, 9, 3), (8, 0, 11, 5), (1, 7, 10, 4)\}$. It is clear that, up to relabeling, $\core(N_{12})$ equals $F_7$ with one element replaced by a parallel class of size 3. The desired result now follows from Lemma \ref{lem:wheelglue} with $N = N_{12}$ and $\mathcal{F}_N = \mathcal{F}^+ = \mathcal{F}$, where we make use of the fact that, to get a 3-connected matroid, at least two elements from the parallel class must be among the deleted elements.
\end{proof}

The following is now obvious:

\begin{corollary}\label{cor:N12graft}
  Let $M$ be a 3-connected, Fano-fragile matroid having a minor isomorphic to $N_{12}$. Then $M$ has a graft representation as in Figure \ref{fig:N12graft}(c), where each of the dashed edges may or may not be present.  
\end{corollary}

\section{$\{U_{2,5}, U_{3,5}\}$-fragile matroids}\label{sec:U25}

In this section we will prove Theorem \ref{thm:H5characterization} and Corollary \ref{thm:U2characterization}.
Recall that we defined $\mathcal{M}_5$ to be the set of strictly $\{U_{2,5},U_{3,5}\}$-fragile matroids that are representable over the partial field $\mathbb{H}_5$, and that we defined $\mathcal{M}_2$ to be the set of strictly $\{U_{2,5},U_{3,5}\}$-fragile matroids that are representable over the partial field $\mathbb{U}_2$.

Our proof, which consists mostly of computer checking, will proceed as follows. First, we review some theory regarding the partial fields $\mathbb{H}_5$ and $\mathbb{U}_2$. In particular, we will use the relationship with the six-fold product ring of $\GF{5}$. 
Then we will enumerate all 3-connected members of $\mathcal{M}_5$ with up to $9$ elements. After that, we apply the Splitter Theorem (in two instances) and Lemmas \ref{lem:fanlemma} and \ref{lem:wheelglue} (in four instances) to finish off our result.

\subsection{The Hydra-5 partial field}

In \cite{PZ08conf}, the Hydra-5 partial field is defined as
\begin{align*}
	\hydra_5 = (\Q(\alpha,\beta,\gamma),\langle & -1,\alpha,\beta,\gamma,\alpha-1,\beta-1,\gamma-1,\alpha-\gamma,
	   \gamma-\alpha\beta,\\
	& (1-\gamma)-(1-\alpha)\beta\rangle),
\end{align*}
where $\alpha$, $\beta$, $\gamma$ are indeterminates. The \emph{fundamental elements} (i.e. those elements $x \in \hydra_5$ such that $1-x \in \hydra_5$) are
\begin{align*}
 \assoc\Big\{&1,\alpha,\beta,\gamma,\tfrac{\alpha\beta}{\gamma},\tfrac{\alpha}{\gamma},\tfrac{(1-\alpha)\gamma}{\gamma-\alpha},\tfrac{(\alpha-1)\beta}{\gamma-1},\tfrac{\alpha-1}{\gamma-1},\tfrac{\gamma-\alpha}{\gamma-\alpha\beta},
	\tfrac{(\beta-1)(\gamma-1)}{\beta(\gamma-\alpha)}, \tfrac{\beta(\gamma-\alpha)}{\gamma-\alpha\beta},\tfrac{(\alpha-1)(\beta-1)}{\gamma-\alpha},\\
	&
	\tfrac{\beta(\gamma-\alpha)}{(1-\gamma)(\gamma-\alpha\beta)},\tfrac{(1-\alpha)(\gamma-\alpha\beta)}{\gamma-\alpha},\tfrac{1-\beta}{\gamma-\alpha\beta}
	\Big\},
\end{align*}
where $\assoc(X) = \bigcup_{x\in X} \{x, 1-x, 1/(1-x), x/(x-1), (x-1)/x, 1/x\}$. Our interest in this partial field is based on the following result (which is \cite[Lemma 5.17]{PZ08conf}):

\begin{theorem}\label{lem:hydra5}
	Let $M$ be a 3-connected matroid.
	\begin{enumerate}
		\item If $M$ has at least 5 inequivalent representations over $\GF{5}$, then $M$ is representable over $\hydra_5$;
		\item If $M$ has a $U_{2,5}$- or $U_{3,5}$-minor and $M$ is representable over $\hydra_5$, then $M$ has at least 6 inequivalent representations over $\GF{5}$.
	\end{enumerate}
\end{theorem}


Rather than working with $\mathbb{H}_5$ directly, we will use the \emph{product ring} $\bigotimes_{i=1}^6 \GF{5}$ whose elements are 6-tuples of $\GF{5}$-elements, with componentwise addition and multiplication. We will use \cite[Lemma 5.8]{PZ08conf}:

\begin{lemma}\label{lem:58}
	Let $A$ be a matrix over $\bigotimes_{i=1}^k \GF{5}$, for $k \geq 3$, such that the projections onto a single coordinate are pairwise inequivalent. Let $p \in \bigotimes_{i=1}^k \GF{5}$ be such that $p \neq (0,0,\ldots, 0), (1,1,\ldots,1)$, and some entry of $p$ occurs in at least three positions. Then $p$ is not among the cross ratios of $A$.
\end{lemma}

It follows that we only need to be concerned in our computations with extensions in $\bigotimes_{i=1}^6 \GF{5}$ with the set of fundamental elements restricted to those in which each of $2, 3, 4$ appears exactly twice. Note that, up to permuting the coordinates, $U_{2,5}$ has a \emph{unique} representation over $\bigotimes_{i=1}^6 \GF{5}$, and that $U_{2,5}$ is a \emph{stabilizer} for the class of $\GF{5}$-representable matroids. That is, if $M$ is a 3-connected, $\GF{5}$-representable matroid with a $U_{2,5}$-minor, then the representation of $U_{2,5}$ uniquely determines the representation of $M$.

Since there is a partial-field homomorphism $\phi:\mathbb{U}_2 \to \mathbb{H}_5$, it follows that $\mathcal{M}_2 \subseteq \mathcal{M}_5$. Hence most of our efforts below are focused on $\mathcal{M}_5$. 

For more details, we refer to \cite{PZ08conf} and \cite{PZ10report}.

\subsection{The small matroids}

Figures \ref{fig:H5_5}--\ref{fig:H5_9} have geometric representations of all 3-connected matroids in $\mathcal{M}_5$ that have at most 9 elements. In the case of dual pairs, we have usually drawn only one of the two. Note that $M_{9,3}$, $M_{9,4}$, $M_{9,6}$, and their duals all have a minor in $\{X_8, Y_8, Y_8^*\}$. In Figure \ref{fig:M71} we included a labeled version of $M_{7,1}$ and two of the matroids containing it as a minor, $M_{8,6}$ and $M_{9,7}$.

\begin{figure}[htbp]
	\centering
		\includegraphics[scale=1.2]{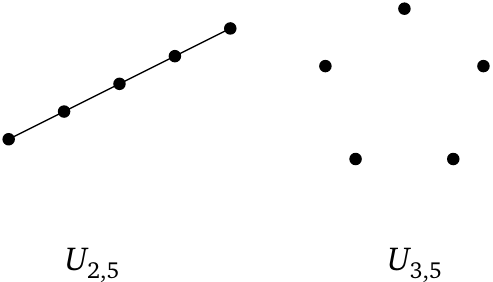}
	\caption{The 3-connected matroids in $\mathcal{M}_5$ on 5 elements.}
	\label{fig:H5_5}
\end{figure}

\begin{figure}[htbp]
	\centering
		\includegraphics[scale=1.2]{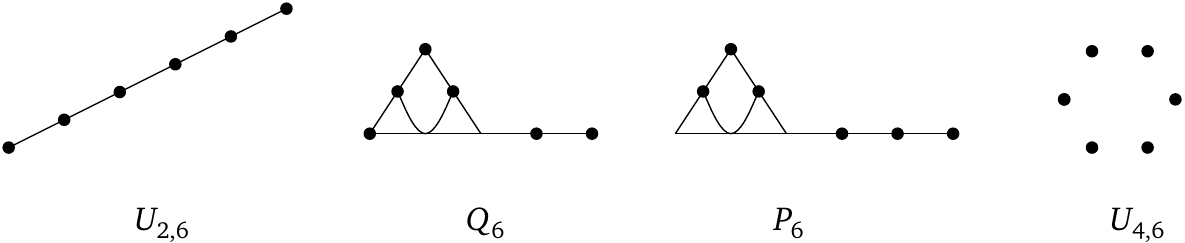}
	\caption{The 3-connected matroids in $\mathcal{M}_5$ on 6 elements.}
	\label{fig:H5_6}
\end{figure}

\begin{figure}[htbp]
	\centering
		\includegraphics[scale=1.2]{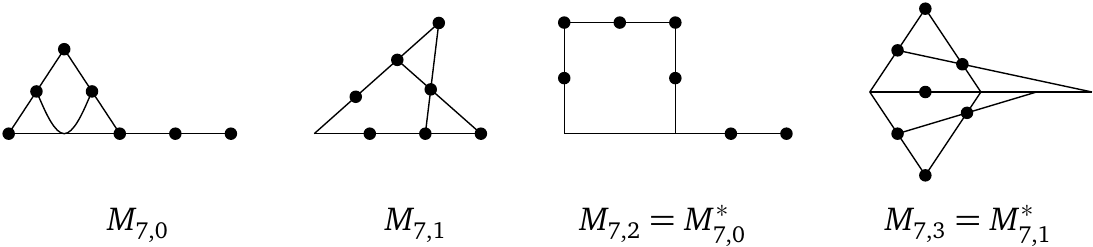}
	\caption{The 3-connected matroids in $\mathcal{M}_5$ on 7 elements.}
	\label{fig:H5_7}
\end{figure}

\begin{figure}[htbp]
	\centering
		\includegraphics[scale=1.2]{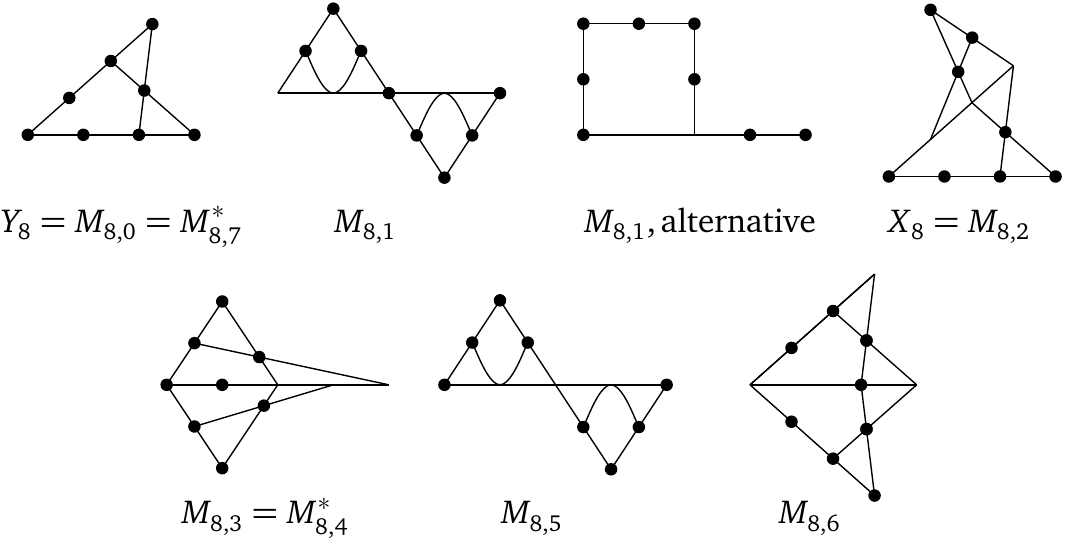}
	\caption{The 3-connected matroids in $\mathcal{M}_5$ on 8 elements.}
	\label{fig:H5_8}
\end{figure}

\begin{figure}[htbp]
	\centering
		\includegraphics[scale=1.2]{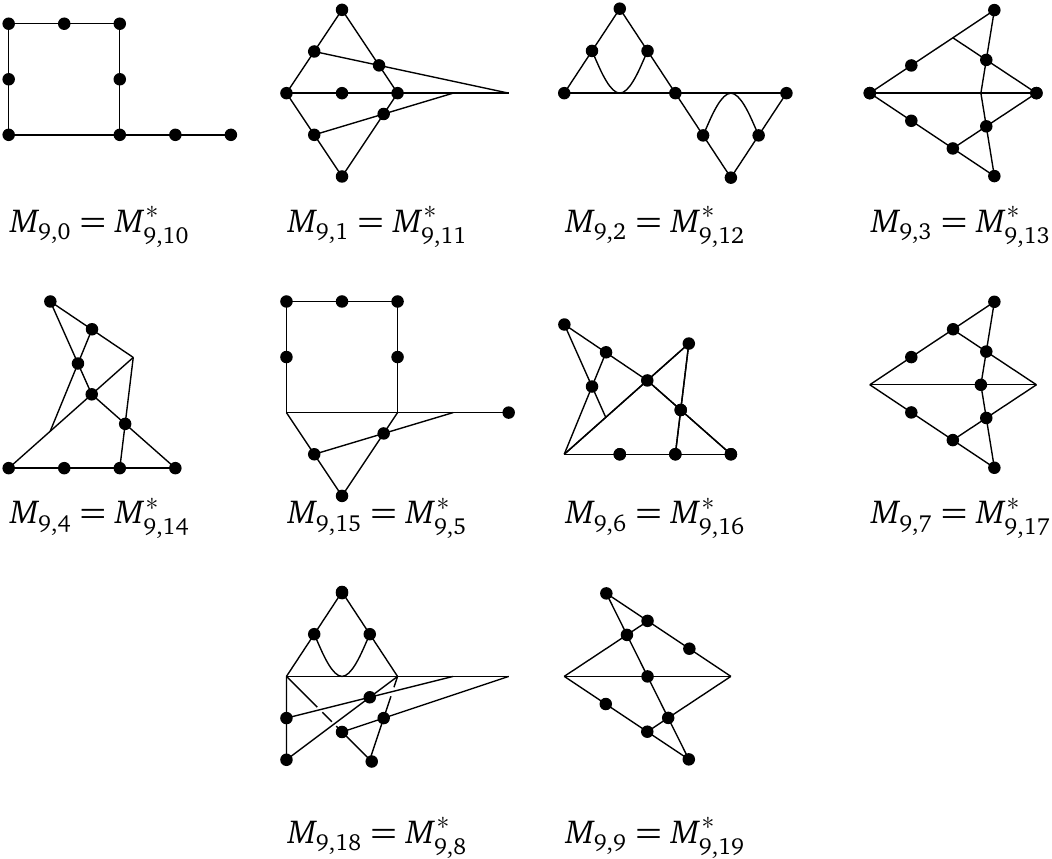}
	\caption{The 3-connected matroids in $\mathcal{M}_5$ on 9 elements.}
	\label{fig:H5_9}
\end{figure}

\begin{figure}[htbp]
  \centering
    \includegraphics[scale=1]{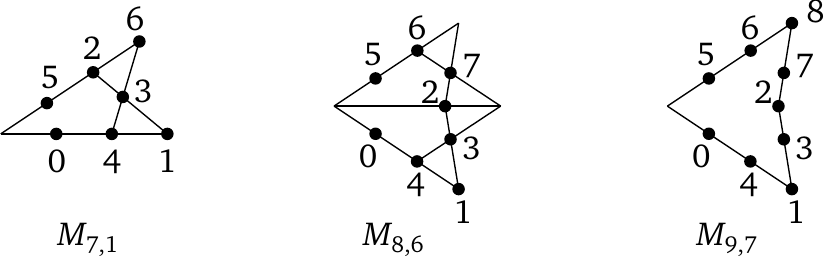}
  \caption{The matroids $M_{7,1}$, $M_{8,6}$, and $M_{9,7}$. In the right-most diagram, the 2-point lines were omitted to emphasize the fan $(1,3,2,7,8)$. Note that the labelings of $M_{8,6}$ and $M_{9,7}$ differ from the labelings in the appendix.}
  \label{fig:M71}
\end{figure}

\newpage

\subsection{The main proof}

We prove Theorem \ref{thm:H5characterization}, which we repeat here for convenience:

\begin{theorem}\label{thm:H5characterizationrepeat}
	Let $M' \in \mathcal{M}_5$ be 3-connected. Then $M'$ is isomorphic to a matroid $M$ for which one of the following holds:
	\begin{enumerate}
		\item\label{it:H51} $M$ has one of $X_8, Y_8, Y_8^*$ as a minor;
		\item\label{it:H52} $M \in \{U_{2,6}, U_{3,6}, U_{4,6}, P_6, M_{9,9}, M_{9,9}^*\}$;
		\item\label{it:H53} $M$ or $M^*$ can be obtained from $U_{2,5}$ (with groundset $\{a,b,c,d,e\}$) by gluing wheels to $(a,c,b), (a,d,b), (a,e,b)$;
		\item\label{it:H54} $M$ or $M^*$ can be obtained from $U_{2,5}$ (with groundset $\{a,b,c,d,e\}$) by gluing wheels to $(a,b,c), (c,d,e)$;
		\item\label{it:H55} $M$ or $M^*$ can be obtained from $M_{7,1}$ (labeled as in Figure \ref{fig:M71}) by gluing a wheel to $(1,3,2)$.
	\end{enumerate}
\end{theorem}

\begin{proof}
	By examining the matroids in the previous subsection, it is straightforward to check that the theorem holds for all matroids on at most 9 elements. For instance, to obtain $Q_6$ we glue a rank-3 wheel to the triangle $(a,c,b)$, and to obtain $M_{8,6}$ we glue a rank-3 wheel to the triangle $(1,3,2)$ of $M_{7,1}$ (using the labeling of Figure \ref{fig:M71}).
	
	Let $M \in \mathcal{M}_5$ be a minor-minimal counterexample. We assume $M$ has at least one of the matroids $\{M_{9,0}, \ldots, M_{9,19}\}$ as a minor. We go through these case by case. Since $M$ does not satisfy $\eqref{it:H51}$, $M$ can not have a minor in $\{M_{9,3}, M_{9,13}, M_{9,4}, M_{9,14}, M_{9,6}, M_{9,16}\}$. 
	
	Assume $M$ has $M_{9,9}$ as a minor. By Lemma A.3, $M$ is isomorphic to $M_{9,9}$, a contradiction. This rules out $M_{9,19}$ as well.
	
	Next, suppose that $M$ has $M_{9,18}$ as a minor. By Lemma A.4, $M$ is a fan-extension of $M_{9,18}$ with respect to $\mathcal{F} = \{(4,6,8), (3,2,5), (0,1,7)\}$. Let $M' := \core(M_{9,18})$. Then $E(M') = \{a_1,b_1,c_1,a_2,b_2,c_2,a_3,b_3,c_3\}$, with $\{a_1, a_2, a_3\}$ and $\{c_1, c_2, c_3\}$ parallel classes, and $\si(M') \cong U_{2,5}$. Hence Lemma \ref{lem:wheelglue}, with $N = M_{9,18}$ and $\mathcal{F}_N = \mathcal{F}^+ = \mathcal{F}$, implies that $M$ satisfies Case \eqref{it:H53} (where we use that, in order to get a 3-connected matroid, at least two elements from each parallel class must be deleted anyway). Using duality, it follows that we may assume from now on that $M$ has no minor isomorphic to $M_{9,8}$ or $M_{9,18}$.
	
	Next, suppose that $M$ has $M_{9,7}$ as a minor. By Lemma A.5, $M$ is a fan-extension of $M_{9,7}$ with respect to $\mathcal{F} = \{(6,5,2,3,8)\}$ (labeled as in the appendix). $\core(M_{9,7})$ is, up to relabeling, equal to $M_{7,1}$, so Lemma \ref{lem:wheelglue}, with $N = M_{9,7}$ and $\mathcal{F}_N = \mathcal{F}^+ = \mathcal{F}$, implies that $M$ satisfies Case \eqref{it:H55}. Using duality, it follows that we may assume from now on that $M$ has no minor isomorphic to $M_{9,7}$ or $M_{9,17}$.
	
	Next, suppose that $M$ has $M_{9,15}$ as a minor. By Lemma A.6, $M$ is a fan-extension of $M_{9,15}$ with respect to $\mathcal{F} = \{(7,1,0,6,8),(3,2,5)\}$.  Let $M' := \core(M_{9,15})$. Then $E(M') = \{a_1,b_1,c_1,a_2,b_2,c_2,4\}$, with $\{a_1, a_2\}$, $\{c_1, c_2\}$ parallel pairs, and  $\si(M') \cong U_{2,5}$. It follows from Lemma \ref{lem:wheelglue} again that $M$ satisfies Case \eqref{it:H53}. Using duality, it follows that we may assume from now on that $M$ has no minor isomorphic to $M_{9,15}$ or $M_{9,5}$.
	
	Next, suppose that $M$ has $M_{9,2}$ as a minor. By Lemma A.7, $M$ is a fan-extension of $M_{9,2}$ with respect to the fans $\mathcal{F} = \{(6,5,2,3), (4,0,1,7)\}$. Let $M' := \core(M_{9,2})$. Then $E(M') = \{a_1,b_1,c_1,a_2,b_2,c_2\}$, with $\{c_1 ,c_2\}$ a parallel class, and  $\si(M') \cong U_{2,5}$. It follows from Lemma \ref{lem:wheelglue} again that $M$ satisfies Case \eqref{it:H54}. Using duality, it follows that we may assume from now on that $M$ has no minor isomorphic to $M_{9,2}$ or $M_{9,12}$.
	
	Next, suppose that $M$ has $M_{9,1}$ as a minor. By Lemma A.8, it follows that $M = M_{9,1}$. Using duality, it follows that we may assume from now on that $M$ has no minor isomorphic to $M_{9,1}$ or $M_{9,11}$.
	
	Finally, suppose that $M$ has $M_{9,0}$ as a minor. By Lemma A.9, either $M$ has $M_{8,5}$ as a minor or $M$ is a fan-extension of $M_{9,0}$ with respect to $\mathcal{F} = \{(1,3,2,5,6,7,8)\}$. In the former case, $M$ must have one of $M_{9,2}, M_{9,4}, M_{9,6}, M_{9,9}, M_{9,12}, M_{9,14}, M_{9,16}$, or $M_{9,19}$ as a minor, all of which were covered before. In the latter case, we argue as before to show that $M$ satisfies Case \eqref{it:H53}. Using duality, $M$ can have neither of $M_{9,0}$ and $M_{9,10}$ as a minor, and this completes the proof.	
\end{proof}

Let us spend a few words on the proof of Corollary \ref{thm:U2characterization}, which we repeat here for convenience.

\begin{corollary}\label{thm:U2characterizationrepeat}
	Let $M' \in \mathcal{M}_2$ be 3-connected. Then $M'$ is isomorphic to a matroid $M$ for which one of the following holds:
	\begin{enumerate}
		\item\label{it:U21} $M$ has one of $X_8, Y_8, Y_8^*$ as a minor;
		\item\label{it:U22} $M \in \{M_{9,9}, M_{9,19}\}$;
		\item\label{it:U23} $M$ or $M^*$ can be obtained from $U_{2,5}$ (with groundset $\{a,b,c,d,e\}$) by gluing wheels to $(a,c,b), (a,d,b), (a,e,b)$;
		\item\label{it:U24} $M$ or $M^*$ can be obtained from $U_{2,5}$ (with groundset $\{a,b,c,d,e\}$) by gluing wheels to $(a,b,c), (c,d,e)$;
		\item\label{it:U25} $M$ or $M^*$ can be obtained from $M_{7,1}$ (labeled as in Figure \ref{fig:M71}) by gluing a wheel to $(1,3,2)$.
	\end{enumerate}
\end{corollary}

\begin{proof}[Sketch of proof]
  Gluing wheels to triangles is done through generalized parallel connection of a matroid with a wheel. Since the latter is regular, this operation preserves representability over a partial field. Since $U_{2,5}$ is $\mathbb{U}_2$-representable, so are all the matroids in Cases \eqref{it:U23}, \eqref{it:U24}. One can check that each of $X_8, Y_8, Y_8^*, M_{9,9}, M_{9,19}, M_{7,1}$ is representable over $\mathbb{U}_2$, since these matroids are generalized $\Delta-Y$ reducible to $U_{2,5}$ (see \cite{OSV00}). 
  
  Conversely, $U_{2,6}$ is \emph{not} representable over $\mathbb{U}_2$, and therefore neither are $P_6, U_{4,6}$ (which are $\Delta-Y$ reducible to $U_{2,6}$).  Hence the result follows.
\end{proof}

\subsection{Some additional checks}
The matroid $X_8$ has a 4-element segment, $S$, and a 4-element cosegment, $C$. The main result in \cite{CMWZ} states that all matroids in $\mathcal{M}_5$ having a minor in $\{X_8,Y_8,Y_8^*\}$ can be obtained from $X_8$ through a \emph{path sequence}: repeated generalized $\Delta-Y$ exchanges on $S$ and $C$, possibly with parallel or series extensions beforehand. Additionally, one can glue a wheel onto appropriate 3-element subsets of $S$ or $C$, as outlined above.

\begin{lemma}\label{lem:pathsequence}
  Let $M \in \mathcal{M}_5$ be 3-connected with $|E(M)| \leq 12$, such that $M$ has a minor in $\{X_8,Y_8,Y_8^*\}$. Then $M$ can be obtained from a path sequence.
\end{lemma}

Not all matroids obtained from a path sequence have one of the minors listed above, but the following is true:

\begin{lemma}\label{lem:pathseqminor}
  Let $M \in \mathcal{M}_5$ be obtained from a path sequence, with $|E(M)| \leq 12$. Then $M$ has a minor in $\{X_8,Y_8,Y_8^*,M_{8,6}\}$.
\end{lemma}

The proofs of these are in the appendix, as Lemmas A.11 and A.10, respectively.


%
%

\bibliographystyle{plain}
\bibliography{../matbib2012}

\appendix

\section{The code}
The code below was run on SageMath 6.5. It should work with newer versions of SageMath too. In this report we divided the code into a few sections. Ready-to-run files can be obtained from \url{http://www.math.lsu.edu/~svanzwam/pdf/fragilecomputations.zip}

\subsection{Product rings}
SageMath Version 6.5 has no support for product rings. The code below is a fairly minimal implementation. We followed the documentation at \url{http://www.sagemath.org/doc/thematic_tutorials/coercion_and_categories.html}. The following is the content of the file \verb\product_ring.py\.

\begin{lstlisting}
r"""
Product rings
"""
import sage
from sage.rings.ring import Ring
from sage.rings.all import ZZ
from sage.structure.element import RingElement
from sage.structure.unique_representation import UniqueRepresentation
from sage.structure.coerce_maps import CallableConvertMap

class ProductRingElement(RingElement):
    """
    Elements of a product ring.
    
    Addition and multiplication happen componentwise.
    """
    def __init__(self, data, parent=None):
        """
        We do very little type checking. It is assumed that ``data`` contains
        an iterable of elements of the appropriate subrings of the appropriate
        length.
        """
        if parent is None:
            raise ValueError, "The parent must be provided"
        self._data = tuple(data)
        RingElement.__init__(self,parent)

    def _repr_(self):
        """
        Return a string representation.
        """
        return "(" + ", ".join(x._repr_() for x in self._data) + ")"

    def __cmp__(self, other):
        """
        Comparison. We rely on Python's built-in tuple comparison.
        """
        return cmp(self._data, other._data)

    def _add_(self, other):
        """
        Add ``self`` to ``other``.
        """
        C = self.__class__
        z = zip(self._data, other._data)
        return C(tuple(zz[0]._add_(zz[1]) for zz in z), parent=self.parent())
        
    def _sub_(self, other):
        """
        Subtract ``other`` from ``self``.
        """
        C = self.__class__
        z = zip(self._data, other._data)
        return C(tuple(zz[0]._sub_(zz[1]) for zz in z), parent=self.parent())

    def _mul_(self, other):
        """
        Multiply ``self`` with ``other``.
        """
        C = self.__class__
        z = zip(self._data, other._data)
        return C(tuple(zz[0]._mul_(zz[1]) for zz in z), parent=self.parent())

    def _div_(self, other):
        """
        Divide ``self`` by ``other``.
        """
        C = self.__class__
        z = zip(self._data, other._data)
        return C(tuple(zz[0]._div_(zz[1]) for zz in z), parent=self.parent())
        
    def __iter__(self):
        r"""
        Return an iterator of the entries of ``self``.
        """
        for x in self._data:
            yield x
        
class ProductRing(UniqueRepresentation, Ring):
    r"""
    The product ring of a finite number of rings, with elementwise addition
    and multiplication.
    """
    Element = ProductRingElement
    
    def __init__(self,rings, base=ZZ, category=None):
        r"""
        INPUT:
        
        - ``rings`` -- a tuple of rings.
        - ``base`` (default: ``ZZ``) -- an underlying base. Not really used,
          but Sage requires it.
        - ``category`` (optional) -- a category, possibly more descriptive
          than ``Rings()``
        """
        from sage.categories.rings import Rings

        if not all(R.is_ring() for R in rings):
            raise TypeError("Expected a tuple of rings as input.")
        self._rings = tuple(rings)
        Ring.__init__(self, base=base, category=category or Rings())

    def _repr_(self):
        """
        Return a string representation.
        """
        return "Product ring: (" + ", ".join(R._repr_() for R in
                                             self._rings) + ")"

    def _element_constructor_(self, *args, **kwds):
        """
        Construct an element of the ring out of a variety of inputs.
        
        Supported is either an iterable of inputs, one for each of the
        subrings, or a single input that gets passed to all subrings.
        """
        if len(args) == len(self._rings):
            z = zip(self._rings, args)
        elif len(args) == 1:
            try:
                if len(args[0]) != len(self._rings):
                    raise TypeError  # also if args[0] has no len()
                z = zip(self._rings, args[0])
            except TypeError:
                z = zip(self._rings, [args[0] for x in self._rings])
        else:
            raise TypeError("Wrong number of ring elements")
        return self.element_class(tuple(zz[0](zz[1]) for zz in z),
                                  parent=self, **kwds)
        
    def _coerce_map_from_(self, S):
        """
        Coercion helps when operating on elements from different rings.
        
        This method makes it possible to write ``p + 1`` for an element ``p``
        of the product ring (provided it works for each of the subrings).
        """
        if all(R.has_coerce_map_from(S) for R in self._rings):
            return CallableConvertMap(S, self,
                        lambda x: self.element_class(
                               tuple(R(x) for R in self._rings), parent=self),
                        parent_as_first_arg=False)

    def __pow__(self,n):
        r"""
        Return the ``n``-th power of self as a vector space.
        """
        from sage.modules.free_module import FreeModule
        return FreeModule(self,n)

    def is_finite(self):
        """
        Determine if the ring is finite.
        """
        return all(R.is_finite() for R in self._rings)

    def is_field(self):
        """
        Determine if the ring happens to be a field.
        
        Since elements of the form ``(0, q)`` are neither zero nor invertible,
        this should return ``False`` unless the product contains but one ring.
        """
        return len(self._rings) == 1
        
    def __len__(self):
        """
        Return the number of elements of the product ring. This will only work
        if all underlying rings are finite.
        """
        return reduce(operator.mul, [len(R) for R in self._rings], 1)
\end{lstlisting}
The notebook containing all our computations starts on the next page.



\section*{Supplementary material}

# Supporting computations

The following two lines load advanced functionality (notably the **setprint** function) and new functionality not standard in Sage (the **ProductRing** class). When running the code, you will need to adjust the path in the second line.

```
from sage.matroids.advanced import *
load /path/to/product_ring.py
```

# Two classes of $\{F_7, F_7^*\}$-fragile matroids

Note that we will abbreviate "binary, $\{F_7, F_7^*\}$-fragile with a minor isomorphic to one of $F_7, F_7^*$" to *Fano-fragile*. We start with the fragility test. We use the fact that the matroids we work with are binary. Hence, not having a minor isomorphic to $F_7$ or $F_7^*$ is equivalent to being regular.

We also introduce to functions to compute the set of deletable, respectively contractible elements of a Fano-fragile matroid.

```
def is_Fano_fragile(M):
    """
    Test if M is Fano-fragile, i.e. for each element e, either M \ e
    or M / e has no Fano- and no Fano-dual-minor.
    """
    for e in M.groundset():
        MR=Matroid(M.delete(e),regular=True, check=False)
        if not MR.is_valid():  # Not regular
            MR=Matroid(M.contract(e),regular=True, check=False)
            if not MR.is_valid():  # Not regular
                return False
    return True

def delset(M):
    """
    Return the set of elements whose deletion preserves a Fano- or
    Fano-dual-minor.
    """
    return [e for e in M.groundset() if not Matroid(M.delete(e),
                            regular=True, check=False).is_valid()]

def conset(M):
    """
    Return the set of elements whose contraction preserves a Fano- or
```

```
        Fano-dual-minor.
        """
    return [e for e in M.groundset() if not Matroid(M.contract(e),
                                regular=True, check=False).is_valid()]
```

Next up, we define the two matroids that need to be tested.

```
N11 = Matroid(reduced_matrix=Matrix(GF(2),[[1,1,0,0,1,1],
        [1,1,1,0,0,0],[0,1,1,1,0,0],[0,0,1,1,1,0],[1,0,0,1,1,1]]))

N12 = Matroid(reduced_matrix=Matrix(GF(2),[[1,1,0,1,1,1],
        [1,0,1,1,1,1],[0,1,1,1,1,1],[1,0,0,1,0,0],[0,1,0,0,1,0],
        [1,1,0,1,1,0]]))
```

# The matroid $N_{11}$

The matroid $N_{11}$ is a 3-connected, single-element extension of $R_{10}$. As such, it is not regular, but it is $\{F_7, F_7^*\}$-fragile.

```
R10 = BinaryMatroid(matroids.named_matroids.R10().representation())
print N11.has_minor(R10)
print N11.is_3connected()
```
```
    True
    True
```
```
print "Deletable elements: ", delset(N11)
print "Contractible elements: ", conset(N11)
```
```
    Deletable elements:  [0, 1, 3, 4, 5, 9]
    Contractible elements:  [2, 6, 7, 8, 10]
```

$N_{11}$ is the *unique* 3-connected Fano-fragile extension of $R_{10}$:

```
S = get_nonisomorphic_matroids([M for M in
R10.linear_extensions(simple=True) if is_Fano_fragile(M)])
print len(S)
print S[0].is_isomorphic(N11)
```
```
    1
    True
```

$N_{11}$ has the following triangles and triads:

```
print "Triangles: "
setprint([C for C in N11.circuits() if len(C) == 3])
```

```
print "Triads: "
setprint([C for C in N11.cocircuits() if len(C) == 3])
    Triangles:
    [{3, 9, 10}, {1, 5, 10}, {0, 4, 10}]
    Triads:
    []
```

In order to prove Theorem 3.1, we need to show that all we can do is grow one of the triangles out into a fan. To break ambiguity, we first compute the nonisomorphic 3-connected single-element extensions and coextensions of $N_{11}$. We label the new element by 11.

```
S1 = [M for M in N11.linear_extensions(11, simple=True) if
      is_Fano_fragile(M)]
print len(S1)
    0
```

```
S2 = [M for M in N11.linear_coextensions(11, cosimple=True) if
      is_Fano_fragile(M)]
print len(S2)
    6
```

Note that all coextensions are isomorphic:

```
print len(get_nonisomorphic_matroids(S2))
    1
```

In order to get reproducible results, we are going to select a specific representative. It turns out the coextensions all have a unique triad; we pick the one with triad $\{4, 10, 11\}$.

```
S2p = [M for M in S2 if M.corank([4,10,11]) == 2]
print len(S2p)
N11plus = S2p[0]
    1
```

This matroid, which we call $N_{11}^{+}$, has a 4-element fan $(0, 4, 10, 11)$ or $(0, 10, 4, 11)$:

```
print "Triangles: "
setprint([C for C in N11plus.circuits() if len(C) == 3])
print "Triads: "
setprint([C for C in N11plus.cocircuits() if len(C) == 3])
    Triangles:
    [{0, 4, 10}]
    Triads:
    [{4, 10, 11}]
```

Since $0, 4$ are deletable and $10$ is contractible, we set $\mathcal{F} := \{(0, 10, 4, 11)\}$. We prove the following

**Lemma A.1.** Every 3-connected, Fano-fragile matroid $M$ with a minor isomorphic to $N_{11}^+$ is a fan-extension of $N_{11}^+$ with respect to $\mathcal{F}$.

*Proof.* By Lemma 2.4, it suffices to verify this for all matroids with at most two more elements than $N_{11}^+$. We can save ourselves some time by observing that $N_{11}^+$ is self-dual:

```
N11plus.is_isomorphic(N11plus.dual())
```
    True

We start with the single-element extensions:

```
exList = [M for M in N11plus.linear_extensions(12, simple=True) if
          is_Fano_fragile(M)]
print len(exList)
```
    1

```
N11pp = exList[0]
print "Triangles: "
setprint([C for C in N11pp.circuits() if len(C) == 3])
print "Triads: "
setprint([C for C in N11pp.cocircuits() if len(C) == 3])
```
    Triangles:
    [{0, 4, 10}, {4, 11, 12}]
    Triads:
    [{4, 10, 11}]

This is indeed a fan-extension. By duality we are done with the single-element coextensions too. Next we need to look at the double-extension, double-coextension, extension-coextension, and coextension-extension cases. Again, by duality, we only need to check the first and third.

```
doubleList = [M for M in N11pp.linear_extensions(13, simple=True) if
              is_Fano_fragile(M)]
doubleList.extend([M for M in N11pp.linear_coextensions(13,
                  cosimple=True) if is_Fano_fragile(M)])
print len(doubleList)
```
    2

```
N11ppp = doubleList[0]
print "Triangles: "
setprint([C for C in N11ppp.circuits() if len(C) == 3])
print "Triads: "
setprint([C for C in N11ppp.cocircuits() if len(C) == 3])
```
    Triangles:
    [{0, 4, 10}, {4, 11, 12}]
    Triads:
    [{11, 12, 13}, {4, 10, 11}]

```
N11ppp = doubleList[1]
print "Triangles: "
setprint([C for C in N11ppp.circuits() if len(C) == 3])
print "Triads: "
setprint([C for C in N11ppp.cocircuits() if len(C) == 3])
```

```
Triangles:
[{0, 4, 10}, {4, 11, 12}]
Triads:
[{0, 10, 13}, {4, 10, 11}]
```

These are both fan-extensions, and the proof of the lemma is complete. $\square$

# The matroid $N_{12}$

The matroid $N_{12}$ is also 3-connected. It has three four-element fans:

```
print "Deletable elements: ", delset(N12)
print "Contractible elements: ", conset(N12)
```

```
Deletable elements:  [1, 2, 8, 9, 10, 11]
Contractible elements:  [0, 3, 4, 5, 6, 7]
```

```
print "Triangles: "
setprint([C for C in N12.circuits() if len(C) == 3])
print "Triads: "
setprint([C for C in N12.cocircuits() if len(C) == 3])
```

```
Triangles:
[{1, 2, 8}, {2, 6, 9}, {1, 7, 10}, {0, 8, 11}]
Triads:
[{0, 5, 11}, {3, 6, 9}, {3, 4, 5}, {4, 7, 10}]
```

The 4-element fans are $\mathcal{F} = \{(2, 6, 9, 3), (8, 0, 11, 5), (1, 7, 10, 4)\}$. We prove the following:

**Lemma A.2.** Every 3-connected, Fano-fragile matroid $M$ with a minor isomorphic to $N_{12}$ is a fan-extension of $N_{12}$ with respect to $\mathcal{F}$.

*Proof.* Again we apply Lemma 2.4. And again we get lucky, and duality cuts our job in half:

```
N12.is_isomorphic(N12.dual())
```

```
True
```

```
exList = [M for M in N12.linear_extensions(12, simple=True)
          if is_Fano_fragile(M)]
print len(exList)
```

```
1
```

```
N12p = exList[0]
print "Triangles: "
```

```
setprint([C for C in N12p.circuits() if len(C) == 3])
print "Triads: "
setprint([C for C in N12p.cocircuits() if len(C) == 3])
```

    Triangles:
    [{1, 2, 8}, {2, 6, 9}, {1, 7, 10}, {0, 8, 11}, {5, 11, 12}, {4, 10,
    12}, {3, 9, 12}]
    Triads:
    [{4, 7, 10}, {3, 6, 9}, {0, 5, 11}]

This is clearly a fan-extension. Note that element 12 can be added to each of the three fans. We proceed with the double extension:

```
doubleList = [M for M in N12p.linear_extensions(13, simple=True)
              if is_Fano_fragile(M)]
doubleList.extend([M for M in N12p.linear_coextensions(13,
                                                cosimple=True)
                   if is_Fano_fragile(M)])
print len(doubleList)
```

    4

We go through them one by one:

```
N12pp = doubleList[0]
print "Triangles: "
setprint([C for C in N12pp.circuits() if len(C) == 3])
print "Triads: "
setprint([C for C in N12pp.cocircuits() if len(C) == 3])
```

    Triangles:
    [{2, 6, 9}, {1, 7, 10}, {0, 8, 11}, {5, 11, 12}, {4, 10, 12}, {3, 9,
    12}]
    Triads:
    [{0, 8, 13}, {2, 6, 13}, {1, 7, 13}, {0, 5, 11}, {3, 6, 9}, {4, 7,
    10}]

```
N12pp = doubleList[1]
print "Triangles: "
setprint([C for C in N12pp.circuits() if len(C) == 3])
print "Triads: "
setprint([C for C in N12pp.cocircuits() if len(C) == 3])
```

    Triangles:
    [{1, 2, 8}, {2, 6, 9}, {1, 7, 10}, {0, 8, 11}, {4, 10, 12}]
    Triads:
    [{4, 12, 13}, {3, 5, 13}, {0, 5, 11}, {3, 6, 9}, {4, 7, 10}]

```
N12pp = doubleList[2]
print "Triangles: "
setprint([C for C in N12pp.circuits() if len(C) == 3])
print "Triads: "
```

```
setprint([C for C in N12pp.cocircuits() if len(C) == 3])
```

    Triangles:
    [{1, 2, 8}, {2, 6, 9}, {1, 7, 10}, {0, 8, 11}, {5, 11, 12}]
    Triads:
    [{5, 12, 13}, {3, 4, 13}, {0, 5, 11}, {3, 6, 9}, {4, 7, 10}]

```
N12pp = doubleList[3]
print "Triangles: "
setprint([C for C in N12pp.circuits() if len(C) == 3])
print "Triads: "
setprint([C for C in N12pp.cocircuits() if len(C) == 3])
```

    Triangles:
    [{1, 2, 8}, {2, 6, 9}, {1, 7, 10}, {0, 8, 11}, {3, 9, 12}]
    Triads:
    [{3, 12, 13}, {4, 5, 13}, {0, 5, 11}, {3, 6, 9}, {4, 7, 10}]

Once more, all of these are fan-extensions, and the lemma is proved. $\square$

# The $H_5$-fragile matroids without $X_8, Y_8,$ and $Y_8^*$ minors

We start by constructing the partial field, and the set of relevant cross ratios (cf. [PZ10, Lemma 5.8]). We use the product ring approach, because the hash function of the field of fractions is unreliable in the current version of Sage (5.12). We had to provide our own implementation of the product ring, and followed the documentation in http://www.sagemath.org/doc/thematic_tutorials/coercion_and_categories.html to produce it. Note that we use six copies of $GF(5)$ to emphasize the symmetry.

```
H5 = ProductRing((GF(5), GF(5), GF(5), GF(5), GF(5), GF(5)))
print H5
```

    Product ring: (Ring of integers modulo 5, Ring of integers modulo 5,
    Ring of integers modulo 5, Ring of integers modulo 5, Ring of
    integers modulo 5, Ring of integers modulo 5)

```
H5CrossRatios = set([H5(1)])
H5CrossRatios.update([H5(x) for x in
                        Permutations([2, 2, 3, 3, 4, 4])])
len(H5CrossRatios)
```

    91

```
A = Matrix(H5, [[1,1,1],[1,H5((2,3,4,2,3,4)), H5((3,2,3,4,4,2))]])
U25 = Matroid(reduced_matrix=A)
print U25.cross_ratios().issubset(H5CrossRatios)
print U25.has_line_minor(5)
```

    True
    True

```
def has_U25_or_U35(M):
    d = 1
    while d > 0:
        N = M.simplify().cosimplify()
        d = len(M) - len(N)
        M = N
    if M.full_corank() == 2:
        return len(M) >= 5
    return M.has_line_minor(5)

def is_fragile(M):
    for e in M.groundset():
        if has_U25_or_U35(M \ e) and has_U25_or_U35(M / e):
            return False
    return True
```

# Generation of the 3-connected matroids in the class with at most 9 elements

We will generate the strictly $\{U_{2,5}, U_{3,5}\}$-fragile, $\mathbb{H}_5$-representable, 3-connected matroids with at most 9 elements. Note that *strictly* implies they have one of $U_{2,5}$ and $U_{3,5}$ as a minor. By the Splitter Theorem, then, we can generate all these matroids by repeatedly extending or coextending by a single element. The following method does this.

Note that this code is not very optimized, but since it will finish in a matter of minutes, we felt it unnecessary to make it more complicated. We assume that the elements are labeled $\{0, 1, \ldots, n-1\}$, and will use $n$ for the label of the new element.

```
def plus1(X):
    res = []
    for M in X:
        res.extend(M.linear_extensions(element=len(M), simple=True,
                                       fundamentals=H5CrossRatios))
        res.extend(M.linear_coextensions(element=len(M),
                                cosimple=True, fundamentals=H5CrossRatios))
    res = [M for M in res if is_fragile(M)]
    res = get_nonisomorphic_matroids(res)
    return res
```

We collect our matroids in a dictionary indexed by size:

```
%time
cat = {}
```

```
cat[5] = [U25, U25.dual()]
print "Number of 5-element matroids: ", len(cat[5])
cat[6] = plus1(cat[5])
print "Number of 6-element matroids: ", len(cat[6])
cat[7] = plus1(cat[6])
print "Number of 7-element matroids: ", len(cat[7])
cat[8] = plus1(cat[7])
print "Number of 8-element matroids: ", len(cat[8])
cat[9] = plus1(cat[8])
print "Number of 9-element matroids: ", len(cat[9])
```

```
Number of 5-element matroids:  2
Number of 6-element matroids:  4
Number of 7-element matroids:  4
Number of 8-element matroids:  8
Number of 9-element matroids:  20
CPU time: 103.58 s,  Wall time: 103.62 s
```

Note that different runs of this generation code are not guaranteed to produce the matroids in the same order (though the sets will always be the same). In order to make the code below reproducible, we sort the lists. Experimentation learned us that the tuple $(r, b, c, d)$, where $r$ is the rank, $b$ is the number of bases, $c$ is the number of circuits, and $d$ is the number of cocircuits, uniquely identifies each matroid in our class up to nine elements.

```
def f(M):
    return (M.rank(), M.bases_count(), len(M.circuits()),
len(M.cocircuits()))

for i in cat:
    cat[i] = sorted(cat[i], key=f)
```

Executing the generation code above will take about 5 minutes on a 2011 iMac. If you plan to use the data repeatedly, it is easy to save it:

```
save(cat, "/path/to/catalog.sobj")
```

Reloading is then as simple as the following (but make sure the ProductRing code has been loaded at the start of the document!)

```
cat = load("/path/to/catalog.sobj")
```

## Matroids on five elements

This set is, obviously, $\{U_{2,5}, U_{3,5}\}$. See Figure 4.

In what follows below, we will print, at most, the circuit-closures representation of each matroid. This is because printing the representation matrix is very unwieldy:

```
cat[5][0].representation(reduced=True, labels=False, B=[0,1])
```
```
[(1, 1, 1, 1, 1, 1) (1, 1, 1, 1, 1, 1) (1, 1, 1, 1, 1, 1)]
[(3, 2, 3, 4, 4, 2) (2, 3, 4, 2, 3, 4) (1, 1, 1, 1, 1, 1)]
```

## Matroids on six elements

See Figure 5.

```
print cat[6][0].is_isomorphic(matroids.Uniform(2,6))
print cat[6][1].is_isomorphic(matroids.named_matroids.Q6())
print cat[6][2].is_isomorphic(matroids.named_matroids.P6())
print cat[6][3].is_isomorphic(matroids.Uniform(4,6))
```
```
True
True
True
True
```
```
setprint(cat[6][1].circuit_closures())
```
```
{2: {{0, 1, 4}, {1, 2, 3}}, 3: {{0, 1, 2, 3, 4, 5}}}
```

## Matroids on seven elements

See Figure 6. Note that it is possible that, in a separate run of the code, a different isomorphism class representative will be chosen for a matroid. In that case the output below will differ by a permutation of the elements. For that reason we opted to leave the figures unlabeled.

```
setprint(cat[7][0].circuit_closures())
print cat[7][2].dual().is_isomorphic(cat[7][0])
```
```
{2: {{0, 1, 4, 6}, {1, 2, 3}, {2, 5, 6}}, 3: {{0, 1, 2, 3, 4, 5, 6}}}
True
```
```
setprint(cat[7][1].circuit_closures())
print cat[7][3].dual().is_isomorphic(cat[7][1])
```
```
{2: {{0, 1, 4}, {1, 2, 3}, {2, 5, 6}, {3, 4, 6}}, 3: {{0, 1, 2, 3, 4, 5, 6}}}
True
```
```
setprint(cat[7][2].circuit_closures())
setprint(cat[7][3].circuit_closures())
```
```
{2: {{0, 1, 4}}, 3: {{1, 2, 3, 6}, {2, 3, 4, 5}}, 4: {{0, 1, 2, 3, 4, 5}}}
{3: {{0, 1, 4, 6}, {1, 2, 3, 6}, {0, 2, 4, 5}, {1, 3, 4, 5}}, 4: {{0, 1, 2, 3, 4, 5, 6}}}
```

An important observation is that $U_{2,6}$, $U_{4,6}$, and $P_6$ are sporadic: by the Splitter Theorem and the following check, no 3-connected matroid in our family has any of these as a proper minor:

```
any([M.has_minor(matroids.Uniform(2,6)) or
     M.has_minor(matroids.Uniform(4,6)) or
     M.has_minor(matroids.named_matroids.P6()) for M in cat[7]])
```

    False

## Matroids on eight elements

See Figure 7.

```
setprint(cat[8][0].circuit_closures())
print cat[8][7].dual().is_isomorphic(cat[8][0])
```

    {2: {{4, 5, 7}, {0, 1, 4, 6}, {2, 5, 6}, {1, 2, 3, 7}}, 3: {{0, 1, 2, 3, 4, 5, 6, 7}}}
    True

```
setprint(cat[8][1].circuit_closures())
print cat[8][1].dual().is_isomorphic(cat[8][1])
```

    {2: {{0, 1, 4}, {1, 2, 3}, {2, 5, 6}}, 3: {{1, 2, 3, 5, 6}, {0, 1, 2, 3, 4}, {0, 1, 4, 6, 7}}, 4: {{0, 1, 2, 3, 4, 5, 6, 7}}}
    True

```
setprint(cat[8][2].circuit_closures())
print cat[8][2].dual().is_isomorphic(cat[8][2])
```

    {2: {{0, 1, 4, 6}}, 3: {{0, 2, 3, 5}, {2, 5, 6, 7}, {1, 2, 3, 7}}, 4: {{0, 1, 2, 3, 4, 5, 6, 7}}}
    True

```
setprint(cat[8][3].circuit_closures())
print cat[8][3].dual().is_isomorphic(cat[8][4])
setprint(cat[8][4].circuit_closures())
```

    {2: {{2, 5, 6}, {0, 1, 6}}, 3: {{2, 3, 4, 5, 6}, {0, 1, 2, 5, 6}, {0, 1, 4, 6, 7}, {1, 2, 3, 7}}, 4: {{0, 1, 2, 3, 4, 5, 6, 7}}}
    True
    {2: {{0, 1, 4}, {1, 2, 3}, {2, 5, 6}}, 3: {{1, 2, 3, 5, 6}, {0, 1, 2, 3, 4}, {3, 4, 6, 7}}, 4: {{0, 1, 2, 3, 4, 5, 6, 7}}}

```
setprint(cat[8][5].circuit_closures())
print cat[8][5].dual().is_isomorphic(cat[8][5])
```

    {2: {{0, 1, 4}, {2, 5, 6}}, 3: {{2, 3, 4, 5, 6}, {0, 1, 4, 6, 7}, {1, 2, 3, 7}}, 4: {{0, 1, 2, 3, 4, 5, 6, 7}}}
    True

```
setprint(cat[8][6].circuit_closures())
print cat[8][6].dual().is_isomorphic(cat[8][6])
```

    {2: {{0, 4, 6}, {2, 5, 6}}, 3: {{1, 2, 3, 7}, {0, 1, 4, 6, 7}, {1,

```
       3, 4, 5}, {0, 2, 4, 5, 6}}, 4: {{0, 1, 2, 3, 4, 5, 6, 7}}}
       True
```

```
X8 = cat[8][2]
Y8 = cat[8][0]
Y8d = cat[8][7]
```

## Matroids on nine elements

See Figure 8. We start by identifying the dual pairs:

```
print cat[9][0].is_isomorphic(cat[9][10].dual())
print cat[9][1].is_isomorphic(cat[9][11].dual())
print cat[9][2].is_isomorphic(cat[9][12].dual())
print cat[9][3].is_isomorphic(cat[9][13].dual())
print cat[9][4].is_isomorphic(cat[9][14].dual())
print cat[9][5].is_isomorphic(cat[9][15].dual())
print cat[9][6].is_isomorphic(cat[9][16].dual())
print cat[9][7].is_isomorphic(cat[9][17].dual())
print cat[9][8].is_isomorphic(cat[9][18].dual())
print cat[9][9].is_isomorphic(cat[9][19].dual())
```
```
       True
       True
       True
       True
       True
       True
       True
       True
       True
       True
```

Next we print the circuit closures, which provide all necessary information to draw the geometric representations.

```
for i in range(10):
    print "Matroid 9, ", i
    setprint(cat[9][i].circuit_closures())
```
```
       Matroid 9,  0
       {2: {{0, 1, 4, 8}, {1, 2, 3}, {2, 5, 6}, {6, 7, 8}}, 3: {{2, 5, 6,
       7, 8}, {1, 2, 3, 5, 6}, {0, 1, 2, 3, 4, 8}, {0, 1, 4, 6, 7, 8}}, 4:
       {{0, 1, 2, 3, 4, 5, 6, 7, 8}}}
       Matroid 9,  1
       {2: {{2, 3, 8}, {1, 7, 8}, {4, 6, 8}, {2, 5, 6}, {0, 1, 6}}, 3: {{2,
       3, 4, 5, 6, 8}, {1, 2, 3, 7, 8}, {0, 1, 2, 5, 6}, {0, 1, 4, 6, 7,
       8}}, 4: {{0, 1, 2, 3, 4, 5, 6, 7, 8}}}
       Matroid 9,  2
```

```
{2: {{0, 1, 4}, {1, 7, 8}, {4, 6, 8}, {2, 5, 6}, {2, 3, 8}}, 3: {{2,
3, 4, 5, 6, 8}, {1, 2, 3, 7, 8}, {0, 1, 4, 6, 7, 8}}, 4: {{0, 1, 2,
3, 4, 5, 6, 7, 8}}}
Matroid 9,  3
{2: {{4, 5, 7}, {0, 1, 6}, {1, 2, 3, 7}}, 3: {{0, 1, 2, 3, 6, 7},
{1, 2, 3, 4, 5, 7}, {0, 1, 4, 6, 8}, {2, 5, 6, 8}}, 4: {{0, 1, 2, 3,
4, 5, 6, 7, 8}}}
Matroid 9,  4
{2: {{0, 1, 4, 6}, {2, 5, 6}, {2, 3, 7}}, 3: {{1, 2, 3, 7, 8}, {4,
5, 7, 8}, {0, 1, 2, 4, 5, 6}, {2, 3, 5, 6, 7}}, 4: {{0, 1, 2, 3, 4,
5, 6, 7, 8}}}
Matroid 9,  5
{2: {{0, 1, 4}, {1, 2, 3}, {2, 5, 6}, {0, 6, 8}, {4, 7, 8}}, 3: {{1,
2, 3, 5, 6}, {0, 1, 2, 3, 4}, {0, 2, 5, 6, 8}, {0, 1, 4, 6, 7, 8}},
4: {{0, 1, 2, 3, 4, 5, 6, 7, 8}}}
Matroid 9,  6
{2: {{4, 5, 7}, {2, 5, 6}, {2, 3, 7}, {0, 1, 6}}, 3: {{1, 2, 3, 7,
8}, {0, 1, 2, 5, 6}, {2, 3, 4, 5, 6, 7}, {0, 1, 4, 6, 8}}, 4: {{0,
1, 2, 3, 4, 5, 6, 7, 8}}}
Matroid 9,  7
{2: {{2, 3, 8}, {0, 4, 6}, {2, 5, 6}, {1, 7, 8}}, 3: {{1, 2, 3, 7,
8}, {0, 1, 4, 6, 7, 8}, {1, 3, 4, 5}, {2, 3, 5, 6, 8}, {0, 2, 4, 5,
6}}, 4: {{0, 1, 2, 3, 4, 5, 6, 7, 8}}}
Matroid 9,  8
{2: {{0, 1, 4}, {4, 6, 8}, {3, 7, 8}, {1, 2, 3}, {2, 5, 6}}, 3: {{1,
2, 3, 5, 6}, {0, 1, 2, 3, 4}, {3, 4, 6, 7, 8}, {1, 2, 3, 7, 8}, {0,
1, 4, 6, 8}, {2, 4, 5, 6, 8}}, 4: {{0, 1, 2, 3, 4, 5, 6, 7, 8}}}
Matroid 9,  9
{2: {{0, 4, 6}, {2, 5, 6}, {1, 3, 8}, {2, 7, 8}}, 3: {{2, 5, 6, 7,
8}, {1, 3, 4, 5, 8}, {0, 1, 4, 6, 7}, {1, 2, 3, 7, 8}, {0, 2, 4, 5,
6}}, 4: {{0, 1, 2, 3, 4, 5, 6, 7, 8}}}
```

The following matroids have a minor in $\{X_8, Y_8, Y_8^*\}$ and can be disregarded from now on:

```
[i for i in range(10) if cat[9][i].has_minor(X8) or
                         cat[9][i].has_minor(Y8) or
                         cat[9][i].has_minor(Y8d)]
```

```
[3, 4, 6]
```

We will deal with the remaining matroids one by one. Using duality, this means that we need to check seven cases. It will be convenient to traverse our list in reverse order.

# $M_{9,9}$ : a splitter.

**Lemma A.3.** The matroid $M_{9,9}$ is a splitter for the class of matroids that are $\mathbb{H}_5$-representable and have no minor in $\{X_8, Y_8, Y_8^*\}$.

*Proof*. We apply the Splitter Theorem:

```
print [M for M in cat[9][9].linear_extensions(simple=True,
                                    fundamentals=H5CrossRatios)
        if is_fragile(M) and not M.has_minor(X8)
                        and not M.has_minor(Y8)
                        and not M.has_minor(Y8d)]
print [M for M in cat[9][9].linear_coextensions(cosimple=True,
                                    fundamentals=H5CrossRatios)
        if is_fragile(M) and not M.has_minor(X8)
                        and not M.has_minor(Y8)
                        and not M.has_minor(Y8d)]
```

```
[]
[]
```

# Fan-extensions of $M_{9,18}$.

We start by determining the fans:

```
def print_fans(M):
    print "Triangles: "
    setprint([C for C in M.circuits() if len(C) == 3])
    print "Triads: "
    setprint([C for C in M.cocircuits() if len(C) == 3])
```

```
M18 = cat[9][18]
print_fans(M18)
print "Deletable elements: "
setprint([e for e in M18.groundset() if (M18 \ e).has_minor(U25)])
```

```
Triangles:
[]
Triads:
[{0, 5, 8}, {4, 6, 8}, {2, 3, 5}, {0, 1, 7}, {3, 4, 7}]
Deletable elements:
[1, 2, 6]
```

The three disjoint triads are $\mathcal{F} = \{(4, 6, 8), (3, 2, 5), (0, 1, 7)\}$, where the ordering of the elements was chosen so the deletable element is in the center. We will show:

**Lemma A.4.** Every 3-connected, $\mathbb{H}_5$-representable, $\{U_{2,5}, U_{3,5}\}$-fragile matroid $M$ with a minor isomorphic to $M_{9,18}$ is a fan-extension of $M_{9,18}$ with respect to $\mathcal{F}$.

*Proof.* We apply Lemma 2.4. First we find all extensions and coextensions by a single element:

```
exList = [M for M in M18.linear_extensions(9, simple=True,
                                    fundamentals=H5CrossRatios)
            if is_fragile(M)]
```

```
exList.extend([M for M in M18.linear_coextensions(9, cosimple=True,
                                         fundamentals=H5CrossRatios)
                   if is_fragile(M)])
print len(exList)
```

2

```
for M in exList:
    print M.rank()
    print_fans(M)
    print " "
```

5
Triangles:
[{1, 7, 9}, {4, 6, 9}, {2, 3, 9}]
Triads:
[{0, 1, 7}, {2, 3, 5}, {4, 6, 8}, {0, 5, 8}]

5
Triangles:
[{6, 8, 9}, {2, 5, 9}, {0, 1, 9}]
Triads:
[{0, 1, 7}, {3, 4, 7}, {2, 3, 5}, {4, 6, 8}]

Both are fan-extensions.

```
doubleList = []
for N in exList:
    doubleList = [M for M in N.linear_extensions(10, simple=True,
                                        fundamentals=H5CrossRatios)
                      if is_fragile(M)]
    doubleList.extend([M for M in N.linear_coextensions(10,
                           cosimple=True, fundamentals=H5CrossRatios)
                          if is_fragile(M)])
    doubleList.extend([M for M in N.linear_extensions(10,
                              simple=True, fundamentals=H5CrossRatios)
                          if is_fragile(M)])
    doubleList.extend([M for M in N.linear_coextensions(10,
                           cosimple=True, fundamentals=H5CrossRatios)
                          if is_fragile(M)])
print len(doubleList)
```

8

```
for M in doubleList:
    print M.rank()
    print_fans(M)
    print " "
```

5
Triangles:

[{6, 8, 9}, {2, 5, 9}, {0, 1, 9}, {1, 7, 10}, {4, 6, 10}, {2, 3,
10}]
Triads:
[{4, 6, 8}, {2, 3, 5}, {0, 1, 7}]

6
Triangles:
[{6, 8, 9}]
Triads:
[{8, 9, 10}, {0, 5, 10}, {4, 6, 8}, {2, 3, 5}, {3, 4, 7}, {0, 1, 7}]

6
Triangles:
[{0, 1, 9}]
Triads:
[{0, 9, 10}, {5, 8, 10}, {4, 6, 8}, {2, 3, 5}, {3, 4, 7}, {0, 1, 7}]

6
Triangles:
[{2, 5, 9}]
Triads:
[{5, 9, 10}, {0, 8, 10}, {4, 6, 8}, {2, 3, 5}, {3, 4, 7}, {0, 1, 7}]

5
Triangles:
[{6, 8, 9}, {2, 5, 9}, {0, 1, 9}, {1, 7, 10}, {4, 6, 10}, {2, 3,
10}]
Triads:
[{4, 6, 8}, {2, 3, 5}, {0, 1, 7}]

6
Triangles:
[{6, 8, 9}]
Triads:
[{8, 9, 10}, {0, 5, 10}, {4, 6, 8}, {2, 3, 5}, {3, 4, 7}, {0, 1, 7}]

6
Triangles:
[{0, 1, 9}]
Triads:
[{0, 9, 10}, {5, 8, 10}, {4, 6, 8}, {2, 3, 5}, {3, 4, 7}, {0, 1, 7}]

6
Triangles:
[{2, 5, 9}]
Triads:
[{5, 9, 10}, {0, 8, 10}, {4, 6, 8}, {2, 3, 5}, {3, 4, 7}, {0, 1, 7}]

Each of these is a fan-extension of $M_{9,18}$, so we are done with $M_{9,18}$. $\square$

```
M8 = cat[9][8]
print_fans(M8)
print "Contractible elements: "
setprint([e for e in M8.groundset() if (M8 / e).has_minor(U25)])
```
```
Triangles:
[{0, 1, 4}, {1, 2, 3}, {2, 5, 6}, {3, 7, 8}, {4, 6, 8}]
Triads:
[]
Contractible elements:
[0, 5, 7]
```

# Fan-extensions of $M_{9,7}$.

We prove the following:

**Lemma A.5.** Every 3-connected, $\mathbb{H}_5$-representable, $\{U_{2,5}, U_{3,5}\}$-fragile matroid $M$ with a minor isomorphic to $M_{9,7}$ but no minor in $\{X_8, Y_8, Y_8^*\}$ is a fan-extension of $M_{9,7}$ with respect to $\mathcal{F} = \{(6, 5, 2, 3, 8)\}$.

*Proof.* We start by observing that this is indeed a fan:

```
M7 = cat[9][7]
print_fans(M7)
```
```
Triangles:
[{2, 5, 6}, {0, 4, 6}, {1, 7, 8}, {2, 3, 8}]
Triads:
[{2, 3, 5}]
```

Once again, we apply Lemma 2.4. We compute all single-element extensions and coextensions:

```
exList = [M for M in M7.linear_extensions(9, simple=True,
                                    fundamentals=H5CrossRatios)
         if is_fragile(M) and not M.has_minor(X8)
                          and not M.has_minor(Y8)
                          and not M.has_minor(Y8d)]
exList.extend([M for M in M7.linear_coextensions(9, cosimple=True,
                                    fundamentals=H5CrossRatios)
              if is_fragile(M) and not M.has_minor(X8)
                               and not M.has_minor(Y8)
                               and not M.has_minor(Y8d)])
print len(exList)
```
```
2
```
```
for M in exList:
    print M.rank()
    print_fans(M)
```

```
print " "
```

```
5
Triangles:
[{2, 5, 6}, {1, 7, 8}, {2, 3, 8}]
Triads:
[{5, 6, 9}, {0, 4, 9}, {2, 3, 5}]

5
Triangles:
[{2, 5, 6}, {0, 4, 6}, {2, 3, 8}]
Triads:
[{3, 8, 9}, {1, 7, 9}, {2, 3, 5}]
```

These are both fan-extensions. Next, we add another extension or coextension:

```
doubleList = []
for N in exList:
    doubleList.extend([M for M in N.linear_extensions(10,
                            simple=True, fundamentals=H5CrossRatios)
                        if is_fragile(M) and not M.has_minor(X8)
                                        and not M.has_minor(Y8)
                                        and not M.has_minor(Y8d)])
    doubleList.extend([M for M in N.linear_coextensions(10,
                            cosimple=True, fundamentals=H5CrossRatios)
                        if is_fragile(M) and not M.has_minor(X8)
                                        and not M.has_minor(Y8)
                                        and not M.has_minor(Y8d)])
print len(doubleList)
```

```
4
```

```
for M in doubleList:
    print M.rank()
    print_fans(M)
    print " "
```

```
5
Triangles:
[{2, 5, 6}, {1, 7, 8}, {2, 3, 8}, {6, 9, 10}, {0, 4, 10}]
Triads:
[{2, 3, 5}, {5, 6, 9}]

6
Triangles:
[{2, 5, 6}, {2, 3, 8}]
Triads:
[{1, 7, 10}, {3, 8, 10}, {2, 3, 5}, {0, 4, 9}, {5, 6, 9}]

5
```

```
Triangles:
[{2, 5, 6}, {0, 4, 6}, {2, 3, 8}, {8, 9, 10}, {1, 7, 10}]
Triads:
[{2, 3, 5}, {3, 8, 9}]

6
Triangles:
[{2, 5, 6}, {2, 3, 8}]
Triads:
[{0, 4, 10}, {5, 6, 10}, {2, 3, 5}, {1, 7, 9}, {3, 8, 9}]
```

These are all fan-extensions, which completes the proof. $\square$

# Fan-extensions of $M_{9,15}$

We prove the following:

**Lemma A.6.** Every 3-connected, $\mathbb{H}_5$-representable, $\{U_{2,5}, U_{3,5}\}$-fragile matroid $M$ with a minor isomorphic to $M_{9,15}$ but no minor in $\{X_8, Y_8, Y_8^*, M_{9,8}, M_{9,18}\}$ is a fan-extension of $M_{9,15}$ with respect to $\mathcal{F} = \{(7, 1, 0, 6, 8), (3, 2, 5)\}$.

*Proof.* We start by observing that these are indeed fans:

```
M15 = cat[9][15]
print_fans(M15)
```

```
Triangles:
[{0, 1, 6}]
Triads:
[{4, 5, 8}, {0, 6, 8}, {2, 3, 5}, {0, 1, 7}, {3, 4, 7}]
```

Once again, we apply Lemma 2.4. We compute all single-element extensions and coextensions:

```
M18d = M18.dual()
exList = [M for M in M15.linear_extensions(9, simple=True,
                                  fundamentals=H5CrossRatios)
         if is_fragile(M) and not M.has_minor(M18)
                          and not M.has_minor(M18d)
                          and not M.has_minor(X8)
                          and not M.has_minor(Y8)
                          and not M.has_minor(Y8d)]
exList.extend([M for M in M15.linear_coextensions(9, cosimple=True,
                                  fundamentals=H5CrossRatios)
             if is_fragile(M) and not M.has_minor(M18)
                              and not M.has_minor(M18d)
                              and not M.has_minor(X8)
                              and not M.has_minor(Y8)
```

```
                                        and not M.has_minor(Y8d)])
print len(exList)
```
```
    2
```
```
for M in exList:
    print M.rank()
    print_fans(M)
    print " "
```
```
    5
    Triangles:
    [{0, 1, 6}, {6, 8, 9}, {2, 5, 9}]
    Triads:
    [{0, 1, 7}, {3, 4, 7}, {2, 3, 5}, {0, 6, 8}]

    5
    Triangles:
    [{0, 1, 6}, {1, 7, 9}, {2, 3, 9}]
    Triads:
    [{0, 1, 7}, {2, 3, 5}, {0, 6, 8}, {4, 5, 8}]
```

These are all fan-extensions. Next, we add another extension or coextension:

```
doubleList = []
for N in exList:
    doubleList.extend([M for M in N.linear_extensions(10,
                        simple=True, fundamentals=H5CrossRatios)
                    if is_fragile(M) and not M.has_minor(M18)
                            and not M.has_minor(M18d)
                            and not M.has_minor(X8)
                            and not M.has_minor(Y8)
                            and not M.has_minor(Y8d)])
    doubleList.extend([M for M in N.linear_coextensions(10,
                        cosimple=True, fundamentals=H5CrossRatios)
                    if is_fragile(M) and not M.has_minor(M18)
                            and not M.has_minor(M18d)
                            and not M.has_minor(X8)
                            and not M.has_minor(Y8)
                            and not M.has_minor(Y8d)])
print len(doubleList)
```
```
    6
```
```
for M in doubleList:
    print M.rank()
    print_fans(M)
    print " "
```
```
    5
    Triangles:
```

```
[{0, 1, 6}, {6, 8, 9}, {2, 5, 9}, {4, 9, 10}, {1, 7, 10}, {2, 3,
10}]
Triads:
[{0, 6, 8}, {2, 3, 5}, {0, 1, 7}]


6
Triangles:
[{0, 1, 6}, {6, 8, 9}]
Triads:
[{8, 9, 10}, {4, 5, 10}, {0, 6, 8}, {2, 3, 5}, {3, 4, 7}, {0, 1, 7}]


6
Triangles:
[{0, 1, 6}, {2, 5, 9}]
Triads:
[{5, 9, 10}, {4, 8, 10}, {0, 6, 8}, {2, 3, 5}, {3, 4, 7}, {0, 1, 7}]


5
Triangles:
[{0, 1, 6}, {1, 7, 9}, {2, 3, 9}, {4, 9, 10}, {6, 8, 10}, {2, 5,
10}]
Triads:
[{0, 6, 8}, {2, 3, 5}, {0, 1, 7}]


6
Triangles:
[{0, 1, 6}, {1, 7, 9}]
Triads:
[{7, 9, 10}, {3, 4, 10}, {4, 5, 8}, {0, 6, 8}, {2, 3, 5}, {0, 1, 7}]


6
Triangles:
[{0, 1, 6}, {2, 3, 9}]
Triads:
[{3, 9, 10}, {4, 7, 10}, {4, 5, 8}, {0, 6, 8}, {2, 3, 5}, {0, 1, 7}]
```

These are all fan-extensions, which completes the proof. $\square$

# Fan-extensions of $M_{9,2}$

We prove the following:

**Lemma A.7.** Every 3-connected, $\mathbb{H}_5$-representable, $\{U_{2,5}, U_{3,5}\}$-fragile matroid $M$ with a minor isomorphic to $M_{9,2}$ but no minor in $\{X_8, Y_8, Y_8^*\}$ is a fan-extension of $M_{9,2}$ with respect to $\mathcal{F} = \{(6, 5, 2, 3), (4, 0, 1, 7)\}$.

*Proof.* We start by observing that these are indeed fans. Note also that elements 6 and 8 can be added to either fan.

```
M2 = cat[9][2]
print_fans(M2)
```
```
Triangles:
[{0, 1, 4}, {2, 5, 6}, {1, 7, 8}, {4, 6, 8}, {2, 3, 8}]
Triads:
[{2, 3, 5}, {0, 1, 7}]
```

Once again, we apply Lemma 2.4. We compute all single-element extensions and coextensions:

```
exList = [M for M in M2.linear_extensions(9, simple=True,
                                    fundamentals=H5CrossRatios)
          if is_fragile(M) and not M.has_minor(X8)
                          and not M.has_minor(Y8)
                          and not M.has_minor(Y8d)]
exList.extend([M for M in M2.linear_coextensions(9, cosimple=True,
                                    fundamentals=H5CrossRatios)
              if is_fragile(M) and not M.has_minor(X8)
                              and not M.has_minor(Y8)
                              and not M.has_minor(Y8d)])
print len(exList)
```
```
3
```
```
for M in exList:
    print M.rank()
    print_fans(M)
    print " "
```
```
5
Triangles:
[{2, 5, 6}, {0, 1, 4}, {1, 7, 8}]
Triads:
[{7, 8, 9}, {0, 1, 7}, {2, 3, 5}]

5
Triangles:
[{0, 1, 4}, {2, 5, 6}, {1, 7, 8}, {2, 3, 8}]
Triads:
[{5, 6, 9}, {0, 4, 9}, {0, 1, 7}, {2, 3, 5}]

5
Triangles:
[{0, 1, 4}, {2, 5, 6}, {2, 3, 8}]
Triads:
[{3, 8, 9}, {0, 1, 7}, {2, 3, 5}]
```

These are all fan-extensions. Next, we add another extension or coextension:

```
doubleList = []
for N in exList:
    doubleList.extend([M for M in N.linear_extensions(10,
                            simple=True, fundamentals=H5CrossRatios)
                        if is_fragile(M) and not M.has_minor(X8)
                                         and not M.has_minor(Y8)
                                         and not M.has_minor(Y8d)])
    doubleList.extend([M for M in N.linear_coextensions(10,
                            cosimple=True, fundamentals=H5CrossRatios)
                        if is_fragile(M) and not M.has_minor(X8)
                                         and not M.has_minor(Y8)
                                         and not M.has_minor(Y8d)])
print len(doubleList)
```

    8

```
for M in doubleList:
    print M.rank()
    print_fans(M)
    print " "
```

    5
    Triangles:
    [{2, 5, 6}, {0, 1, 4}, {1, 7, 8}, {8, 9, 10}, {2, 3, 10}, {4, 6,
    10}]
    Triads:
    [{0, 1, 7}, {2, 3, 5}, {7, 8, 9}]

    6
    Triangles:
    [{2, 5, 6}, {0, 1, 4}, {1, 7, 8}]
    Triads:
    [{3, 9, 10}, {0, 4, 10}, {5, 6, 10}, {2, 3, 5}, {0, 1, 7}, {7, 8,
    9}]

    5
    Triangles:
    [{0, 1, 4}, {2, 5, 6}, {1, 7, 8}, {2, 3, 8}, {4, 9, 10}, {6, 8, 10}]
    Triads:
    [{2, 3, 5}, {0, 1, 7}, {0, 4, 9}]

    5
    Triangles:
    [{0, 1, 4}, {2, 5, 6}, {1, 7, 8}, {2, 3, 8}, {6, 9, 10}, {4, 8, 10}]
    Triads:
    [{2, 3, 5}, {0, 1, 7}, {5, 6, 9}]

    6
    Triangles:
    [{2, 5, 6}, {0, 1, 4}, {2, 3, 8}]

```
Triads:
[{7, 9, 10}, {3, 8, 10}, {2, 3, 5}, {0, 1, 7}, {0, 4, 9}, {5, 6, 9}]

6
Triangles:
[{2, 5, 6}, {0, 1, 4}, {1, 7, 8}]
Triads:
[{3, 9, 10}, {7, 8, 10}, {2, 3, 5}, {0, 1, 7}, {0, 4, 9}, {5, 6, 9}]

5
Triangles:
[{0, 1, 4}, {2, 5, 6}, {2, 3, 8}, {8, 9, 10}, {1, 7, 10}, {4, 6,
10}]
Triads:
[{0, 1, 7}, {2, 3, 5}, {3, 8, 9}]

6
Triangles:
[{2, 5, 6}, {0, 1, 4}, {2, 3, 8}]
Triads:
[{7, 9, 10}, {0, 4, 10}, {5, 6, 10}, {2, 3, 5}, {0, 1, 7}, {3, 8,
9}]
```

These are all fan-extensions, which completes the proof. □

# Fan-extensions of $M_{9,1}$

We prove the following:

**Lemma A.8.** Every 3-connected, $\mathbb{H}_5$-representable, $\{U_{2,5}, U_{3,5}\}$-fragile matroid $M$ with a proper minor isomorphic to $M_{9,1}$ has a minor in $\{X_8, Y_8, Y_8^*, M_{9,15}, M_{9,18}\}$.

*Proof.* This is a simple case check:

```
M1 = cat[9][1]
exList = [M for M in M1.linear_extensions(9, simple=True,
                                fundamentals=H5CrossRatios)
         if is_fragile(M) and not M.has_minor(M15)
                          and not M.has_minor(M18)
                          and not M.has_minor(X8)
                          and not M.has_minor(Y8)
                          and not M.has_minor(Y8d)]
exList.extend([M for M in M1.linear_coextensions(9, cosimple=True,
                                fundamentals=H5CrossRatios)
              if is_fragile(M) and not M.has_minor(M15)
                               and not M.has_minor(M18)
                               and not M.has_minor(X8)
```

```
                              and not M.has_minor(Y8)
                              and not M.has_minor(Y8d)])
print len(exList)
```
    0

Hence all matroids in our class containing $M_{9,1}$ have been covered by previous cases. □

# Fan-extensions of $M_{9,0}$

Finally, we prove the following:

**Lemma A.9.** Every 3-connected, $\mathbb{H}_5$-representable, $\{U_{2,5}, U_{3,5}\}$-fragile matroid $M$ with $M_{9,0}$ as a minor but no minor in $\{X_8, Y_8, Y_8^*, M_{8,5}, M_{9,15}, M_{9,7}\}$ is a fan-extension of $M_{9,0}$ with respect to $\mathcal{F} = \{(1, 3, 2, 5, 6, 7, 8)\}$.

Note the appearance of $M_{8,5}$. We will explain that below the proof.

*Proof.* First we check that the member of $\mathcal{F}$ is indeed a fan:

```
M0 = cat[9][0]
print_fans(M0)
```
    Triangles:
    [{0, 1, 4}, {1, 2, 3}, {2, 5, 6}, {6, 7, 8}, {0, 1, 8}, {0, 4, 8},
    {1, 4, 8}]
    Triads:
    [{2, 3, 5}, {5, 6, 7}]

Once more, we apply Lemma 2.4. We compute all single-element extensions and coextensions:

```
S = [X8, Y8, Y8d, cat[8][5], cat[9][5], cat[9][15], cat[9][7], cat[9]
[17]]
```

```
exList = [M for M in M0.linear_extensions(9, simple=True,
                            fundamentals=H5CrossRatios)
          if is_fragile(M) and not any(M.has_minor(N) for N in S)]
exList.extend([M for M in M0.linear_coextensions(9, cosimple=True,
                              fundamentals=H5CrossRatios)
              if is_fragile(M) and not  any(M.has_minor(N)
                                          for N in S)])
print len(exList)
```
    2

```
for M in exList:
    print M.rank()
    print_fans(M)
    print " "
```
    5

Triangles:
[{2, 5, 6}, {1, 2, 3}, {0, 1, 4}, {6, 7, 8}]
Triads:
[{7, 8, 9}, {0, 4, 9}, {5, 6, 7}, {2, 3, 5}]

5
Triangles:
[{1, 2, 3}, {2, 5, 6}, {6, 7, 8}, {0, 4, 8}]
Triads:
[{1, 3, 9}, {0, 4, 9}, {5, 6, 7}, {2, 3, 5}]

These are all fan-extensions. Next, we add another extension or coextension:

```
doubleList = []
for N in exList:
    doubleList.extend([M for M in N.linear_extensions(10,
                        simple=True, fundamentals=H5CrossRatios)
                    if is_fragile(M) and not any(M.has_minor(N)
                                        for N in S)])
    doubleList.extend([M for M in N.linear_coextensions(10,
                        cosimple=True, fundamentals=H5CrossRatios)
                    if is_fragile(M) and not any(M.has_minor(N)
                                        for N in S)])
print len(doubleList)
```

4

```
for M in doubleList:
    print M.rank()
    print_fans(M)
    print " "
```

5
Triangles:
[{2, 5, 6}, {1, 2, 3}, {0, 1, 4}, {6, 7, 8}, {8, 9, 10}, {0, 4, 10},
{1, 4, 10}, {0, 1, 10}]
Triads:
[{5, 6, 7}, {2, 3, 5}, {7, 8, 9}]

6
Triangles:
[{2, 5, 6}, {1, 2, 3}, {6, 7, 8}]
Triads:
[{0, 9, 10}, {4, 9, 10}, {0, 4, 10}, {1, 3, 10}, {2, 3, 5}, {5, 6,
7}, {0, 4, 9}, {7, 8, 9}]

5
Triangles:
[{1, 2, 3}, {2, 5, 6}, {6, 7, 8}, {0, 4, 8}, {1, 9, 10}, {0, 8, 10},

```
    {4, 8, 10}, {0, 4, 10}]
    Triads:
    [{2, 3, 5}, {5, 6, 7}, {1, 3, 9}]

    6
    Triangles:
    [{2, 5, 6}, {1, 2, 3}, {6, 7, 8}]
    Triads:
    [{0, 9, 10}, {4, 9, 10}, {0, 4, 10}, {7, 8, 10}, {2, 3, 5}, {5, 6, 7}, {0, 4, 9}, {1, 3, 9}]
```

These are all fan-extensions, which completes the proof. □

Now, in order to make sure we have covered all cases for $M_{9,0}$, we must ensure that all extensions and coextensions of $M_{8,5}$ have been considered previously. This is indeed the case:

```
[i for i in range(20) if cat[9][i].has_minor(cat[8][5])]
    [2, 4, 6, 9, 12, 14, 16, 19]
```

# Path sequences

We start by defining the generalized Delta-Y exchange operations:

```
def DY(M, segment, parallels=None):
    # First, construct the parallel extension of M
    if parallels is not None:
        for e in parallels:
            M = M.linear_extension(element=len(M), chain={e: 1})

    # Next, find a suitable representation of the segment:
    BS = M.augment([], segment)
    NBS = [e for e in segment if not e in BS]
    B = M.augment(BS)
    A, row_labels, col_labels = M.representation(reduced=True, B=B, labels=True)

    # construct the Theta matroid
    row_indices = [row_labels.index(e) for e in BS]
    col_indices = [col_labels.index(e) for e in NBS]
    k = len(segment)
    D = Matrix(A.base_ring(), nrows=k, ncols=k)
    D[0,1] = A[row_indices[1],col_indices[0]] /
A[row_indices[0],col_indices[0]]
    D[1,0] = -D[0,1]
    for i in range(k-2):
```

```
            D[i+2,0] = A[row_indices[0],col_indices[i]]
            D[i+2,1] = A[row_indices[1],col_indices[i]]
            D[0,i+2] = A[row_indices[0],col_indices[i]]
            D[1,i+2] = A[row_indices[1],col_indices[i]]

        # construct the ground set of the generalized parallel connection
        tmpE = M.groundset_list()
        delset = []
        E = copy(row_labels)
        for i in row_indices:
            E[i] = newlabel(tmpE)
            tmpE.append(E[i])
            delset.append(E[i])
        E.extend(NBS)
        F = copy(col_labels)
        for i in col_indices:
            F[i] = newlabel(tmpE)
            tmpE.append(F[i])
            delset.append(F[i])
        F.extend(BS)
        E.extend(F)

        # construct the new representation
        AA = Matrix(A.base_ring(), nrows=A.nrows() + k - 2,
ncols=A.ncols()+2)
        AA.set_block(0,0,A)
        AA[row_indices[0], A.ncols()+1] = D[0,1]
        AA[row_indices[1], A.ncols()] = D[1,0]
        for i in range(k-2):
            AA[A.nrows()+i, A.ncols()] = D[i+2,0]
            AA[A.nrows()+i, A.ncols()+1] = D[i+2,1]
        return Matroid(groundset=E,reduced_matrix=AA).delete(delset)

def YD(M, cosegment, series=None):
    return DY(M.dual(), segment=cosegment, parallels=series).dual()
```

The segment and cosegment of $X_8$ are as follows:

```
print X8.rank([0,1,4,6])
print X8.corank([2,3,5,7])
S = [0,1,4,6]
C = [2,3,5,7]
    2
    2
```

Deletable and contractible elements:

```
print "Deletable: ", [e for e in X8.groundset() if has_U25_or_U35(X8 \
e)]
print "Contractible: ", [e for e in X8.groundset() if has_U25_or_U35(X8 /
e)]
```

```
Deletable:  [0, 1, 2, 6]
Contractible:  [3, 4, 5, 7]
```

$X_8$ has an automorphism group whose restriction to $\{0, 1, 6\}$ is $S_3$, and whose restriction to $\{3, 5, 7\}$ is $S_3$:

```
p1 = {0:1, 1:6, 4:4, 6:0, 2:2, 3:7, 5:3, 7:5}
p2 = {0:0, 1:6, 4:4, 6:1, 2:2, 3:5, 5:3, 7:7}
X8.is_isomorphism(X8,p1) and X8.is_isomorphism(X8,p2)
```

```
True
```

The next two functions serve to build all matroids using path sequences.

```
def move(M):
    """
    Return all matroids that can be produced from M by one allowable
move.

    Fan moves are broken up into Delta-Y and Y-Delta moves. We use the
fact that in
    a Delta-Y move, the original elements from X8 survive as an internal
triangle
    or triad of the fan.
    """
    S = [0,1,4,6]
    C = [2,3,5,7]
    safe_elts = [4,2]
    fanbase = [[0,1,4], [0,6,4], [1,6,4], [2,3,5], [2,3,7], [2,5,7]]

    # Find all allowable segments
    segs = []
    if M.rank(S) == 2:
        segs.append(S)
    if M.rank(C) == 2:
        segs.append(C)
    # add allowable triangles
    for T in fanbase:
        if M.rank(T) == 2:
            segs.append(T)

    # Find all allowable cosegments
```

```
    cosegs = []
    if M.corank(S) == 2:
        cosegs.append(S)
    if M.corank(C) == 2:
        cosegs.append(C)
    # add allowable triads
    for T in fanbase:
        if M.corank(T) == 2:
            cosegs.append(T)
    L = []

    # Delta-Y moves
    for S in segs:
        Sp = [e for e in S if e not in safe_elts]
        for X in Subsets(Sp):
            if len(S) == 3 or (len(S) == 4 and len(X) >= 1):
                L.append(DY(M,S,X))

    # Y-Delta moves
    for C in cosegs:
        Cp = [e for e in C if e not in safe_elts]
        for X in Subsets(Cp):
            if len(C) == 3 or (len(C) == 4 and len(X) >= 1):
                L.append(YD(M,C,X))

    return L
```

```
def grow(pathmatroids, newpathmatroids, n=None):
    """
    Take each of the matroids in `newpathmatroids` and add all
pathmatroids that can be obtained through one
    (generalized) Delta-Y move. Throw away results on more than `n`
elements, if these occur.
    """
    L = []
    for M in newpathmatroids:
        L.extend(move(M))
    npm = []
    for N in L:
        if (n is None or N.size() <= n) and not any(M.is_isomorphic(N)
for M in pathmatroids):
            pathmatroids.append(N)
            npm.append(N)
    return pathmatroids, npm
```

We continue growing until no new path matroids are found.

```
pathmatroids = [X8]
newpathmatroids = [X8]

pathmatroids, newpathmatroids = grow(pathmatroids, newpathmatroids, 12)
print len(pathmatroids)
print len(newpathmatroids)
print len([M for M in pathmatroids if M.size() == 8])
print len([M for M in pathmatroids if M.size() == 9])
print len([M for M in pathmatroids if M.size() == 10])
print len([M for M in pathmatroids if M.size() == 11])
print len([M for M in pathmatroids if M.size() == 12])
```

```
13
12
3
4
4
2
0
```

```
pathmatroids, newpathmatroids = grow(pathmatroids, newpathmatroids, 12)
print len(pathmatroids)
print len(newpathmatroids)
print len([M for M in pathmatroids if M.size() == 8])
print len([M for M in pathmatroids if M.size() == 9])
print len([M for M in pathmatroids if M.size() == 10])
print len([M for M in pathmatroids if M.size() == 11])
print len([M for M in pathmatroids if M.size() == 12])
```

```
80
67
4
10
26
26
14
```

```
pathmatroids, newpathmatroids = grow(pathmatroids, newpathmatroids, 12)
print len(pathmatroids)
print len(newpathmatroids)
print len([M for M in pathmatroids if M.size() == 8])
print len([M for M in pathmatroids if M.size() == 9])
print len([M for M in pathmatroids if M.size() == 10])
print len([M for M in pathmatroids if M.size() == 11])
print len([M for M in pathmatroids if M.size() == 12])
```

```
405
325
4
```

        18
        60
        130
        193

```
pathmatroids, newpathmatroids = grow(pathmatroids, newpathmatroids, 12)
print len(pathmatroids)
print len(newpathmatroids)
print len([M for M in pathmatroids if M.size() == 8])
print len([M for M in pathmatroids if M.size() == 9])
print len([M for M in pathmatroids if M.size() == 10])
print len([M for M in pathmatroids if M.size() == 11])
print len([M for M in pathmatroids if M.size() == 12])
```

        958
        553
        4
        20
        84
        253
        597

```
pathmatroids, newpathmatroids = grow(pathmatroids, newpathmatroids, 12)
print len(pathmatroids)
print len(newpathmatroids)
print len([M for M in pathmatroids if M.size() == 8])
print len([M for M in pathmatroids if M.size() == 9])
print len([M for M in pathmatroids if M.size() == 10])
print len([M for M in pathmatroids if M.size() == 11])
print len([M for M in pathmatroids if M.size() == 12])
```

        1540
        582
        4
        20
        92
        348
        1076

```
pathmatroids, newpathmatroids = grow(pathmatroids, newpathmatroids, 12)
print len(pathmatroids)
print len(newpathmatroids)
print len([M for M in pathmatroids if M.size() == 8])
print len([M for M in pathmatroids if M.size() == 9])
print len([M for M in pathmatroids if M.size() == 10])
print len([M for M in pathmatroids if M.size() == 11])
print len([M for M in pathmatroids if M.size() == 12])
```

        1874
        334
        4
        20

93
376
1381

```
pathmatroids, newpathmatroids = grow(pathmatroids, newpathmatroids, 12)
print len(pathmatroids)
print len(newpathmatroids)
print len([M for M in pathmatroids if M.size() == 8])
print len([M for M in pathmatroids if M.size() == 9])
print len([M for M in pathmatroids if M.size() == 10])
print len([M for M in pathmatroids if M.size() == 11])
print len([M for M in pathmatroids if M.size() == 12])
```

1984
110
4
20
93
380
1487

```
pathmatroids, newpathmatroids = grow(pathmatroids, newpathmatroids, 12)
print len(pathmatroids)
print len(newpathmatroids)
print len([M for M in pathmatroids if M.size() == 8])
print len([M for M in pathmatroids if M.size() == 9])
print len([M for M in pathmatroids if M.size() == 10])
print len([M for M in pathmatroids if M.size() == 11])
print len([M for M in pathmatroids if M.size() == 12])
```

1999
15
4
20
93
380
1502

```
pathmatroids, newpathmatroids = grow(pathmatroids, newpathmatroids, 12)
print len(pathmatroids)
print len(newpathmatroids)
print len([M for M in pathmatroids if M.size() == 8])
print len([M for M in pathmatroids if M.size() == 9])
print len([M for M in pathmatroids if M.size() == 10])
print len([M for M in pathmatroids if M.size() == 11])
print len([M for M in pathmatroids if M.size() == 12])
```

1999
0
4
20
93

```
380
1502
```

**Lemma A.10.** Let $M \in \mathcal{M}_5$ be obtained from a path sequence, with $|E(M)| \leq 12$. Then $M$ has a minor in $\{X_8, Y_8, Y_8^*, M_{8,6}\}$.

*Proof:* By finite case check.

```
for M in pathmatroids:
    if not any(M.has_minor(N) for N in [X8, Y8, Y8d, cat[8][6]]):
        print "Counterexample: "
        print M
print "Done"
```

```
Done
```

This completes the proof. $\square$

**Lemma A.11.** Let $M \in \mathcal{M}_5$ be 3-connected with $|E(M)| \leq 12$, such that $M$ has a minor in $\{X_8, Y_8, Y_8^*\}$. Then $M$ can be obtained from a path sequence.

*Proof:* By finite case check. By the above, we can assume $M$ has a minor isomorphic to $M_{9,3}, M_{9,4}, M_{9,6}$, or their duals. We find all 3-connected single-element extensions and coextensions, and check that they occur among the list of path matroids.

```
S = [cat[9][3], cat[9][13], cat[9][4], cat[9][14], cat[9][6], cat[9][16]]
TenEltList = []
for M in S:
    TenEltList.extend(M.linear_extensions(element=len(M), simple=True,
                                          fundamentals=H5CrossRatios))
TenEltList = [M for M in TenEltList if is_fragile(M)]
TenEltList.extend([M.dual() for M in TenEltList])
TenEltList = get_nonisomorphic_matroids(TenEltList)
print len(TenEltList)
```

```
36
```

```
for i in range(36):
    M = TenEltList[i]
    if not any(M.is_isomorphic(N) for N in pathmatroids):
        print "Counterexample: "
        print i
print "Done"
```

```
Done
```

Next, eleven elements:

```
ElevenEltList = []
for M in TenEltList:
```

```
        ElevenEltList.extend(M.linear_extensions(element=len(M), simple=True,
                                                 fundamentals=H5CrossRatios))
ElevenEltList = [M for M in ElevenEltList if is_fragile(M)]
ElevenEltList.extend([M.dual() for M in ElevenEltList])
ElevenEltList = get_nonisomorphic_matroids(ElevenEltList)
print len(ElevenEltList)
```

    90

```
for i in range(len(ElevenEltList)):
    M = ElevenEltList[i]
    if not any(M.is_isomorphic(N) for N in pathmatroids):
        print "Counterexample: "
        print i
print "Done"
```

    Done

Finally, 12 elements:

```
TwelveEltList = []
for M in ElevenEltList:
    TwelveEltList.extend(M.linear_extensions(element=len(M), simple=True,
                                             fundamentals=H5CrossRatios))
TwelveEltList = [M for M in TwelveEltList if is_fragile(M)]
TwelveEltList.extend([M.dual() for M in TwelveEltList])
TwelveEltList = get_nonisomorphic_matroids(TwelveEltList)
print len(TwelveEltList)
```

    255

```
for i in range(len(TwelveEltList)):
    M = TwelveEltList[i]
    if not any(M.is_isomorphic(N) for N in pathmatroids):
        print "Counterexample: "
        print i
print "Done"
```

    Done

This completes the proof, and concludes our computations. □



\end{document}